\documentclass[11 pt,reqno,nocenter]{amsart}

\usepackage{geometry}                
\usepackage{graphicx}
\usepackage{amssymb}
\usepackage{epstopdf}
\numberwithin{equation}{section}

\usepackage{geometry}
\usepackage{graphicx}
\usepackage{amsmath,amssymb,amsthm}
\usepackage{enumerate}
\usepackage{hyperref}
\usepackage[active]{srcltx}
\usepackage[T1]{fontenc}
\usepackage{color}
\newtheorem{theo}{Theorem}

\newtheorem{lem}{Lemma}[section]
\newtheorem{defi}{Definition}[section]
\newtheorem{cor}{Corollary}[section]
\newtheorem{prop}{Proposition}[section]

\newcommand{\eps}{\varepsilon}

\newcommand{\R}{\mathbb{R}}

\begin{document}
\title[]{A Wasserstein gradient flow approach to Poisson-Nernst-Planck equations}
\date{}
\author[]{
  David Kinderlehrer}
    \address{Department of Mathematical Sciences, Carnegie Mellon University, Pittsburgh, PA 15213, USA}
\author[]{ L\'eonard Monsaingeon}
\address{CAMGSD Instituto Superior T\'ecnico, Av. Rovisco Pais 1049-001 Lisbon, Portugal}
\author[]{Xiang Xu}
\address{Department of Mathematics and Statistics, Old Dominion University, 2300 Engineering \& Computational Sciences Bldg
Norfolk, VA 23529, USA}
\email{\tt davidk@andrew.cmu.edu} \email{\tt
leonard.monsaingeon@ist.utl.pt} \email{\tt x2xu@odu.edu}

\maketitle

\begin{abstract}
The Poisson-Nernst-Planck system of equations used to model ionic
transport is interpreted as a gradient flow for the Wasserstein
distance and a free energy in the space of probability measures with finite second
moment. A variational scheme is then set up and is the starting
point of the construction of global weak solutions in a unified
framework for the cases of both linear and nonlinear diffusion. The proof
of the main results relies on the derivation of additional estimates
based on the flow interchange technique developed by Matthes, McCann, and Savar\'e in \cite{MR2581977}.
\end{abstract}

\tableofcontents

\section{Introduction}

The Poisson-Nernst-Planck (PNP) system of equations \cite{MR1818867,
MR1063852} is the principal description of ionic transport of
several interacting species. It has been applied in a number of
contexts ranging from electrical storage devices to molecular
biology, at times coupled to Navier-Stokes or other systems.
The basic system is to
find $u(t,x)\geq 0$ and $v(t,x)\geq 0$ satisfying
\begin{equation}
\begin{aligned}
 \partial_t u  &= \Delta u^m+{\rm div}\left(u\nabla\left(U+\psi\right)\right),\\
 \partial_t v  &= \Delta v^m+{\rm div}\left(v\nabla\left(V-\psi\right)\right), \qquad t\geq 0,\,x\in \R^d,\,d\geq 3, \\
 -\Delta \psi & =  u-v,
\end{aligned}
\label{eq:PNP}
\end{equation}
for some suitable initial conditions $u|_{t=0}=u^0$ and $v|_{t=0}=v^0$.  The unknowns $u$ and $v$ represent the
density of some positively and negatively charged particles. Here $m
\geq 1$ is a chosen fixed nonlinear diffusion exponent. Note that \eqref{eq:PNP}
formally preserves the $L^1$ mass
$$\int_{\R^d}
u(t,x)\mathrm{d}x=\int_{\R^d} u^0(x)\mathrm{d}x \quad {\rm and} \quad\int_{\R^d}
v(t,x)\mathrm{d}x=\int_{\R^d} v^0(x)\mathrm{d}x
$$
for all $t\geq 0$, which physically represents the conservation of total charge of the
individual species. For initial masses
$\int_{\R^d}u^0(x)\mathrm{d}x=\int_{\R^d}v^0(x)\mathrm{d}x$, a simple
rescaling of time and space allows normalization of masses to unity, and
without loss of generality we consider $u(t,x),v(t,x)$ to be
probability densities.
The external potentials $U(x)$ and $V(x)$ are prescribed and
sufficiently smooth. $\psi(x)$ from the Gauss Law is the
self-consistent electrostatic potential created by the two charge
carriers according to the last equation in \eqref{eq:PNP}. The first
two equations in \eqref{eq:PNP} are called Nernst-Planck equations
and describe electro-diffusion and electrophoresis according to the
Fick and Kohlrausch laws, respectively, while the last equation in
\eqref{eq:PNP} corresponds to the electrostatic Poisson law. The
boundary condition for this coupling equation will always be
understood in the sense of the Newtonian potential:
we shall always implicitly
write
\begin{equation}
 \label{eq:def_Newton_potential}
-\Delta \psi=u-v\quad \Leftrightarrow \quad
\psi=G*(u-v)=(-\Delta)^{-1}(u-v),
\end{equation}
where
\begin{equation*}
\qquad G(x):=\frac{1}{d(d-2)\omega_d} | x | ^{2-d}
\end{equation*}
is the fundamental solution of $-\Delta$ in $\R^d$ and  $\omega_d$ the volume of the unit ball in $\R^d$ (for $d\geq
3$).

In some PNP models, an extra background doping profile $C(x)$ is considered resulting in the modified potential
 $-\Delta \psi' = u -v +C$. With suitable assumptions on $C(x)$ this can be easily
eliminated replacing the external potentials $U$ by $U+ \psi_C$ and $V$ by $V-\psi_C$, where $\psi_C=(-\Delta)^{-1}C$.

There is a vast literature on well-posedness and long time behavior
of the system \eqref{eq:PNP}. We refer to \cite{MR1777308,
MR1325799, MR1259205, MR1452478} and references therein for bounded
domains, and \cite{MR1770444, MR1842428, MR1866628, MR2395521,
MR2445259} for the whole space problem. Different from these papers,
our contribution is to show that the system of equations
\eqref{eq:PNP}, governing drift, diffusion and reaction of charged
species, possesses a gradient flow structure as viewed on the metric
space of probability measures on $\R^d$ endowed with the quadratic
Wasserstein distance, in essence the weak-$\ast$ topology (see also
\cite{MR2776123, MR3076364, MR3164863} for various discussions).
Therefore, variational methods may be introduced to prove global
existence of weak solutions (see definition below). Motivation for
this in terms of energy dissipation is offered below. Also, we
provide a unified framework both for linear $m=1$ and nonlinear
$m>1$ diffusions. To the best of our knowledge well-posedness in the
whole space for $m>1$ usually requires either high integrability of
the initial data or initial gradient regularity. We would like to
stress that we need here no such hypotheses and that our result
merely requires some low initial integrability, defined in terms of
the diffusion exponent $m\geq 1$ and dimension $d\geq 3$ only (which
we think is sharp, see discussion after the proof of
Proposition~\ref{prop:weak_Euler-Lagrange}
page~\pageref{rmk:initial_integrability}). Also, we do not need
compatibility conditions between $m\geq 1$ and $d\geq 3$. This is in
contrast with the Patlak-Keller-Segel chemotaxis models
\cite{MR2383208,BCKKLL}, for which critical mass phenomena may occur
depending
 on whether $m$ is larger or smaller than the scale-invariant exponent $m_*(d)=2-\frac 2d$, see e.g.
 \cite{MR3040679} for the critical parabolic-parabolic case. This is mainly due to the fact that the self-induced drifts
  are repulsive here, while they are self-attractive in the Keller-Segel models, thus leading to aggregation and blow-up in finite time.


It was shown in \cite{MR1617171, MR1842429} that certain scalar diffusion equations
can be interpreted as gradient flows in metric spaces and the literature concerning
this issue is steadily growing (see \cite{MR2401600} and references
therein). It is thus a question of great interest to apply such ideas
to study systems of equations. In contrast, there are only a
few examples for systems. For related studies, we refer to
\cite{MR2383208, BCKKLL,MR3040679,MR2368900, MR2461815, MR3044141,Matthes_Zinsl,Zinsl}, where the energy functional
is involved with the Wasserstein metric and existence theorems using
a minimizing movement scheme for corresponding evolution problems
are presented.

In  \cite{MR1617171}, the linear Fokker-Planck
equation $$\partial_t\rho=\sigma\Delta\rho+\operatorname{div}(\rho\nabla\varphi)$$ is regarded as the gradient flow of a free energy consisting of the Boltzmann entropy with a potential  $\varphi$,
$$
\mathcal{F}(\rho)=\int_{\R^d}(\varphi \rho +\sigma \rho\log\rho)\,\mathrm{d}x,
$$
with respect to the quadratic Wasserstein metric. There may be many
Lyapunov functions associated to a differential equation. The result
of \cite{MR1617171}  means that dissipation for the free energy
$\mathcal{F}$ determines the Fokker-Planck Equation.   The same idea
was later employed in \cite{MR1842429} and, in addition, to derive
long-time asymptotics for the Porous Medium Equation (PME). Since
the system \eqref{eq:PNP} can be viewed as two Fokker-Planck
equations (when $m=1$) or Porous Medium equations (when $m>1$) in
$u,v$ coupled by means of a Poisson kernel, we are motivated to
extend these ideas to study our coupled system. Inspired by
\cite{MR1842429} and \cite{MR1964483}, we shall in fact discover
that the PNP system can be seen as a gradient flow driven by the
free energy
\begin{equation}
\label{eq:def_total_energy}
 \begin{aligned}
 & \mathcal E(u,v):= \int_{\R^d} \Big(u\log u+v\log v+u U+v V+\frac{|\nabla\psi|^2
}{2}\Big)\,\mathrm{d}x & \mbox{ if } m=1, \\
& \mathcal E(u,v):= \int_{\R^d} \Big(\frac{u^m}{m-1} + \frac{v^m}{m-1} +u U+v
V+\frac{|\nabla\psi|^2}{2}\Big)\,\mathrm{d}x & \mbox{ if } m>1.
\end{aligned}\end{equation}
We further motivate this approach informally by discussing the
relationship between the dissipation relation and the weak-$\ast$
topology in terms of the Wasserstein-Rubinstein-Kantorovich
distance, or simply the Wasserstein distance. We follow
\cite{MR3164863} and consider for illustration the case  $m=1.$ Set
\begin{equation*}
\begin{aligned}
\label{eq:intro_1}
&\varphi(u,v) = u \log u + v \log v + uU + vV + \frac{1}{2} |\nabla \psi|^2, \ \mbox{so that} \\
&\mathcal  E(u,v) = \int_{\R^d} \varphi(u,v) \mathrm{d}x.
\end{aligned}
\end{equation*}
Given a process or an evolution  $(u(t),v(t))$, during an interval  $(T,T+h)$ the change in energy is
\begin{equation}
\label{eq:intro_2}
\mathcal E(u,v)\Big|_{T+h} - \int_T^{T+h}\int_{\R^d}
\frac{\mathrm{d}}{\mathrm{d}t}\varphi(u,v) \mathrm{d}x\mathrm{d}t =
\mathcal E(u,v)\Big|_T
\end{equation}
This, \eqref{eq:intro_2}, is the dissipation equality or inequality and the density of the middle term
\begin{equation}
\label{eq:intro_3}
\mathcal D =  - \int_{\R^d} \frac{\mathrm{d}}{\mathrm{d}t}\varphi(u,v)
\mathrm{d}x
\end{equation}
is the dissipation density along the trajectory. Writing
\begin{equation}
\label{eq:intro_4}
\frac{\mathrm{d}}{\mathrm{d}t}\varphi(u,v) = \varphi_u
\frac{\mathrm{d}u}{\mathrm{d}t} + \varphi_v
\frac{\mathrm{d}v}{\mathrm{d}t},
\end{equation}
we must ascribe a meaning to
$$
  \frac{\mathrm{d}u}{\mathrm{d}t}, \frac{\mathrm{d}v}{\mathrm{d}t}
$$
to render the system dissipative, that is, so that
\eqref{eq:intro_3} is positive. To begin we calculate the terms in
\eqref{eq:intro_4}. Keeping in mind \eqref{eq:def_Newton_potential},
one checks that
$$
 \frac{\delta}{\delta u} \Big(\frac{1}{2} |\nabla \psi|^2 \Big) = \psi \ \mbox{and} \
 \frac{\delta}{\delta v} \Big(\frac{1}{2} |\nabla \psi|^2 \Big) = -\psi,
$$
which leads to
\begin{equation*}
\label{eq:intro_5}
\varphi_u (u,v) = \log u + U + \psi + 1\  \mbox{and} \ \varphi_v  = \log v + V - \psi + 1.
\end{equation*}
Let us now employ the Poisson-Nernst-Planck equations \eqref{eq:PNP}. Substituting into  \eqref{eq:intro_3}  and integrating by parts gives
$$
\mathcal D = \int_{\R^d} \left\{ \Big|\frac{\nabla u}{u} + \nabla(U + \psi)\Big|^2 u + \Big|\frac{\nabla v}{v} + \nabla(V - \psi)\Big|^2 v \right\}  \mathrm{d}x
$$
Introduce
\begin{equation}
\label{eq:intro_6}
\begin{aligned}
& w =  -\Big(\frac{\nabla u}{u} +  \nabla(U + \psi)\Big) \ \mbox{ and } \ \omega = -\Big(\frac{\nabla v}{v} + \nabla(V - \psi)\Big) \ \mbox{ so that} \\
& u_t + \mbox{div}(wu) = 0 \ \mbox{ and} \ v_t + \mbox{div}(\omega v) = 0.
\end{aligned}
\end{equation}
We have that
\begin{equation*}
\label{eq:intro_7}
\int_T^{T+h} \mathcal D \mathrm{d}t = \int_T^{T+h} \int_{\R^d} |w|^2 u
\mathrm{d}x \mathrm{d}t + \int_T^{T+h} \int_{\R^d} |\omega|^2 v
\mathrm{d}x \mathrm{d}t
\end{equation*}
where the pairs $(u,w), (v,\omega)$ satisfy the continuity equations \eqref{eq:intro_6}. This represents a trial in $d_W,$ the
quadratic Wasserstein metric, using the Benamou-Brenier formulation \cite{MR1738163}, where
\begin{equation*}
\begin{aligned}
&\frac{1}{h} d_W(u\big|_{T+h},u\big|_T)^2 = \inf \int_T^{T+h} \int_{\R^d} |w|^2u \mathrm{d}x \mathrm{d}t \ \mbox{and}
\\ &\frac{1}{h} d_W(v\big|_{T+h},v\big|_T)^2 = \inf \int_T^{T+h} \int_{\R^d} |\omega|^2v \mathrm{d}x \mathrm{d}t \\
& \mbox{ taken over all pairs} \ (u, w) \ \mbox{and} \ (v,\omega) \ \mbox{satisfying the continuity equations} \\
& \mbox{} u_t + \mbox{div}(wu) = 0 \ \mbox{ and} \ v_t + \mbox{div}(\omega v) = 0, \ \mbox{and}\\
& \mbox{the initial and terminal conditions.}
\end{aligned}
\end{equation*}
The calculation shows that the dissipation relation \eqref{eq:intro_6} for $\mathcal E$  and the PNP system
are closely related to the Wasserstein distance. We could write, in fact,
\begin{equation*}
\label{eq:intro_9}
\frac{1}{h} d_W(u\big|_{T+h},u\big|_T)^2 + \frac{1}{h}
d_W(v\big|_{T+h},v\big|_T)^2 + \mathcal E(u,v)\big|_{T+h}  \leq \mathcal E(u,v)\big|_T
\end{equation*}
suggesting an implicit scheme which leads to a gradient flow. This is nearly correct. As is well known,
a factor of $1/2$ must be inserted; see below \eqref{def:JKO_scheme}.

We now turn to the precise formulation. Denoting $\mathcal{P}(\R^d)$ the set of Borel probability measures
on $\R^d$ with finite second moments and $d_W$ the
quadratic Wasserstein distance as before, the
underlying space will be here $(u,v)\in \mathcal{P}(\R^d) \times
\mathcal{P}(\R^d)$ and will inherit a natural differential
structure from that of $(\mathcal{P}(\R^d),d_W)$ - see
section~\ref{section:wass_grad_flow_structure} for details. The
total free energy \eqref{eq:def_total_energy} is a combination of
the well-known internal (diffusive entropy) and potential energies
for each species, and, although unclear at this stage, the coupling
Dirichlet energy $\frac{1}{2}\int_{\R^d}|\nabla\psi|^2\mathrm{d}x$
falls into the category of so-called interaction energies. See
\cite{MR1964483} for an introduction.

Following \cite{MR1617171}, we shall construct weak solutions $z=(u,v)$ as follows.
Given suitable initial data $z^0=(u^0,v^0)$ and some small time step $h\in (0,1)$
we first construct a discrete sequence $\{z_h^{(n)}\}_{n\in \mathbb{N}}$ solution
to the Jordan-Kinderlehrer-Otto or JKO implicit scheme
\begin{equation}
\begin{aligned}
& z_h^{(0)}=z^0, \\
&         z_h^{(n+1)}\in \underset{\mathcal{K}\times \mathcal{K}}{\mbox{Argmin}}\left\{\frac{1}{2h}d^2\big(\cdot,z_h^{(n)}\big)+\mathcal E(\cdot)\right\}
\end{aligned}
\label{def:JKO_scheme}
\end{equation}

Here $\mathcal E(z) =\mathcal  E(u,v)$ is the total free energy
\eqref{eq:def_total_energy}, $d^2$ is the (squared) distance on the
product space inherited from $d_W$, and
$\mathcal{K}\subset \mathcal{P}$ the set of admissible minimizers
defined later on.
%
As is classical by now, one obtains interpolating solutions
$\{z_h(t)\}_{h}=\{u_h(t),z_h(t)\}_{h}$ defined for all $t\geq 0$,
piecewise constant in time, and satisfying a coupled system of two
Euler-Lagrange equations. We shall then prove that as $h \rightarrow
0$ one recovers a weak solution
$(u(t),v(t))=z(t)=\displaystyle\lim_{h\rightarrow 0} z_h(t)$ of
\eqref{eq:PNP}. There are several challenges in this program.

In handling the $\int_{\R^d}|\nabla\psi|^2\,\mathrm{d}x$ coupling
term, some intrinsic difficulties arise due both to the specific
Poisson kernel and to the nonscalar setting. First, as neither the
external potentials nor $G$ are convex the free energy $\mathcal E$ is not
displacement convex in the sense of McCann \cite{MR1451422} and we
cannot simply apply the standard procedures as in \cite{MR2401600}.
Secondly, due to the singular nature of the kernel
$G(x)=C/|x|^{d-2}$ both the existence of minimizers in
\eqref{def:JKO_scheme} and derivation of the corresponding
Euler-Lagrange equations become delicate, see in particular the
discussions in Proposition \ref{prop:Poisson_IBP} and Proposition
\ref{prop:weak_Euler-Lagrange} for details. In order to tackle this
issue we used the "flow interchange" technique that originates in
\cite{MR2581977} and was later used in \cite{MR3040679, MR3044141}
to obtain some integrability improvement and gradient regularity of
the minimizers, see Proposition \ref{prop:discrete_Lp_propagation}
and Proposition \ref{prop:discrete_gradient_np_estimate} below. The
highlight of the argument is the propagation of initial
$L^p(\mathbb{R}^d)$ regularity established in Proposition
\ref{prop:discrete_Lp_propagation}. In addition to being technically
essential here, this propagation of initial regularity allowed us to
obtain a natural $L^{\infty}(\mathbb{R}^d)$ estimate (see Theorems
\ref{theo:main_theo}-\ref{theo:main_theo_linear}), which to the best
of our knowledge was unknown for the PNP system in the whole space.
We also believe that the very same argument could be employed for
similar problems in order to show propagation of initial regularity,
which is usually a delicate point in the mass transport framework.


The rest of the paper is organized as follows. In Section
\ref{section:wass_grad_flow_structure} we recall well-known facts in
optimal transport theory and briefly describe the differential
structure of the product space. We then formally derive the
Wasserstein gradient flow structure of the system \eqref{eq:PNP} and
state the main existence results. In
Section~\ref{section:study_energy} we study the relevant energy
functionals, and establish improved regularity of their minimizers.
In Section~\ref{section:minimizing_scheme} we fix a time step $h>0$
small enough and consider the minimizing scheme. We obtain
approximate discrete solutions $\{u_h,v_h\}_{h}$ and derive the
corresponding Euler-Lagrange equations. In
Section~\ref{section:CV_to_weak_solution} we take the limit $h\to 0$
and show that the convergence $(u,v)=\lim\limits_{h\to 0}(u_h,v_h)$ is
strong enough to retrieve a weak solution. This last section also
contains the proof of the main theorems.
\subsection*{Notation Convention.}

Unless otherwise specified, $\langle \cdot, \cdot\rangle$ and  $\cdot$ denote
inner product of elements in $\mathbb{R}^d$, $\mathcal{P}$ denotes
$\mathcal{P}(\R^d)$, and $\mathcal{P}^{ac}$ denotes
$\mathcal{P}^{ac}(\R^d)$. If clear from the context we shall often
omit the subscripts $m=1$ or $m>1$. If $1\leq p\leq \infty$, we
denote by $p'=\frac{p}{p-1}$ the conjugate Lebesgue exponent. 


\section{Formal Wasserstein gradient flow}
\label{section:wass_grad_flow_structure}

From now on we assume that the external potentials are quadratic at infinity, i-e
\begin{equation}
 \label{eq:structural_potentials}
\begin{aligned}
&C_1|x|^2\leq U(x),V(x)\leq C_2(1+|x|^2)\\
& \ |\nabla U(x)|,|\nabla V(x)|\leq C_3(1+|x|)\\
& \ \|\Delta U\|_{L^{\infty}(\R^d)},\|\Delta V\|_{L^{\infty}(\R^d)}\leq C_4,
\end{aligned}
\end{equation}
for some generic positive constants $C_i$. Note that $U,V$ need not
be uniformly convex as is often assumed, so that we allow here
multiple wells. Although these assumptions on the potentials could
be weakened we assume here strict confinement $C_1>0$ for the ease
of exposition, and we do not seek optimal generality.

We also introduce the admissible set
\begin{equation}
\label{eq:def_admissible_set_K}
\mathcal{K}:=\left\{\begin{array}{ll}
 \mathcal{K}_1:=\mathcal{P}^{ac}(\R^d)\cap L^1\log L^1(\R^d)&\text{if }m=1,\\
 \mathcal{K}_m:=\mathcal{P}^{ac}(\R^d)\cap L^m(\R^d)&\text{if }m>1,
 \end{array}\right.
\end{equation}
and for reasons that shall become clear later on we shall always
consider initial data
\begin{equation}
\label{eq:initial_data}
u^0,v^0\in\mathcal{K}\cap L^{r_0}(\R^d)
\qquad \mbox{ for some }r_0>\max\{m,2d/(d+1)\}.
\end{equation}
Essentially $u^0,v^0\in L^{m}\cap L^{2d/(d+2)}(\R^d)$ ensures that the initial energy $\mathcal (u^0,v^0)$ is finite,
while $L^{2d/(d+1)}(\R^d)$ regularity will ensure that the self-induced drifts $u\nabla\Psi,v\nabla\Psi\in L^1(\R^d)$
for all times. We believe that $r=\max\{m,2d/(d+1)\}$ should be admissible, but for technical compactness issues we
 have to assume here slightly better $L^{r_0}(\R^d)$ integrability for some $r_0>r$ arbitrarily close. See the
 proof of Theorem~\ref{theo:uh->u_a.e.} later on for details.

For $\rho,u,v\geq 0$ let us define the usual Boltzmann entropy
\begin{equation*}
\label{def: Boltzmann entropy}
\mathcal{H}(\rho):=\int_{\R^d} \rho\log\rho\,\mathrm{d}x,
\end{equation*}
the diffusion energy
\begin{equation}
\label{eq:def_energy_diff}
\begin{aligned}
& \mathcal E_{\rm diff}(u,v):= \int_{\R^d}\big( u\log u + v\log v \big)\,\mathrm{d}x \  {\rm if} \ m = 1 \ {\rm and}\\
& \mathcal E_{\rm diff}(u,v):= \frac{1}{m-1}\int_{\R^d} \big( u^m+v^m \big)\,\mathrm{d}x \  {\rm if} \ m > 1
\end{aligned}
\end{equation}
and the external potential energy
\begin{equation}
\mathcal E_{\rm ext}(u,v):=\int_{\R^d} (u U+v V)\,\mathrm{d}x.
\label{eq:def_energy_external}
\end{equation}
Note that, with our assumptions, $\mathcal E_{\rm diff},\mathcal E_{\rm ext}$ are finite for all
$(u,v)\in \mathcal{K}\times \mathcal{K}$. For $\psi=(-\Delta)^{-1}(u-v)=G*(u-v)$ we define now the coupling
energy
\begin{equation}
\mathcal E_{\rm cpl}(u,v) := \frac{1}{2}\int_{\R^d}|\nabla\psi|^2\mathrm{d}x,
%
\label{eq:def_energy_coupling}
\end{equation}
which is the energy of the self-induced electric
potential. Note that, at least formally,
\begin{align}
\int_{\R^d} |\nabla\psi|^2\mathrm{d}x&=\int_{\R^d}(-\Delta
\psi)\psi\,\mathrm{d}x=\int_{\R^d}(u-v)\,G*(u-v)\,\mathrm{d}x\nonumber\\
&=\iint_{\R^d\times\R^d}[u-v](x)G(x-y)[u-v](y)\,\mathrm{d}x\mathrm{d}y
\label{eq:formal_IBP}
\end{align}
falls into the category of interaction energies $$\iint_{\R^d\times\R^d}\rho(x)K(x,y)\rho(y)\mathrm{d}x\mathrm{d}y$$ treated in \cite{MR1964483}.
To sum up, the total free energy $\mathcal E=\mathcal E_{\rm diff}+\mathcal E_{\rm ext}+\mathcal E_{\rm cpl}$ is given by
\eqref{eq:def_total_energy}



For given probability measures $\mu,\nu\in\mathcal{P}(\R^d)$ we denote the squared (quadratic) Wasserstein distance by
\begin{equation*}
\label{eq:W_def}
d_W^2(\mu,\nu)=\inf\limits_{\gamma\in\Gamma(\mu,\nu)}\iint_{\R^d\times
\R^d}|x-y|^2\mathrm{d}\gamma(x,y),
\end{equation*}
where $\Gamma(\mu,\nu)\subset \mathcal{P}(\R^d_x\times\R^d_y)$ is the set of admissible joint distributions
with $x$ and $y$ marginals $\mu,\nu$ respectively. We recall from
\cite{MR1964483} that $(\mathcal{P},d_W)$ is a metric
space and that $d_W$ metrizes the weak convergence of
measures. When $\mu\in \mathcal{P}^{ac}_2$ is moreover absolutely
continuous with respect to the Lebesgue measure $\mathrm{d}\mu(x)\ll
\mathrm{d}x$, the square Wasserstein distance can also be computed
by the Benamou-Brenier theorem \cite{MR1738163} as already noted.
\begin{theo}[Existence of optimal maps, \cite{MR1964483}]
Let $\mu\in \mathcal{P}^{ac}_2(\R^d)$ and $\nu\in
\mathcal{P}^{ac}_2(\R^d)$. There exists a unique optimal transport
map $T=\nabla\varphi\in L^2(\R^d; \mathrm{d}\mu)$ for some convex
function $\varphi$ such that
$$
\nu=T_{\#}\mu:\qquad \forall
f\in\mathcal{C}_c(\R^d),\,\int_{\R^d}f(y)\mathrm{d}\nu(y)=\int_{\R^d}f\circ
T(x)\mathrm{d}\mu(x)
$$
and
$$
d_W^2(\mu,\nu)=\int_{\R^d}|x-T(x)|^2\mathrm{d}\mu(x).
$$
\label{theo:Brenier}
\end{theo}
\noindent
Our interest here is a system so we endow $\mathcal{P}(\R^d)\times\mathcal{P}(\R^d)$ with the natural product distance
$$
d^2(z,z')=d_W^2(u,u')+d_W^2(v,v')
$$
for all $z=(u,v)$ and $z'=(u',v')$ in $\mathcal{P}\times
\mathcal{P}$. It is well known \cite{MR1842429, MR1964483} that
$(\mathcal{P},d_W)$ enjoys a natural differential
structure defined by means of continuity equations, so that our
product space also has the same differential structure.
This permits us to differentiate real-valued functions $\mathcal F$ on the
product space, and defines the corresponding Wasserstein gradient by
the chain rule
$\frac{\mathrm{d}}{\mathrm{d}t}\mathcal F(z_t)=\operatorname{grad}_W
\mathcal F(z_t).\frac{\mathrm{d}z_t}{\mathrm{d}t}$. We show now that
\eqref{eq:PNP} is really the gradient flow
\begin{equation*}
 \frac{\mathrm{d}z}{\mathrm{d}t}=-\operatorname{grad}_W \mathcal E(z),\qquad z(t)=\Big(\begin{array}{c}
        u(t)\\
        v(t)
           \end{array}
\Big).
\label{eq:PNP_gradient_flow_wass}
\end{equation*}
In terms of ordinary calculus of variations, we recall that this is achieved by variation of domain or, in fluid dynamical terms,
by Lagrangian variations, \cite{MR1617171}.
To this end,
let us split for convenience the coupling Poisson equation $-\Delta
\psi=u-v$ as
\begin{equation*}
\psi=\psi_u-\psi_v\qquad\text{with} \quad \left\{
\begin{array}{ccc}
-\Delta \psi_u=u & \Leftrightarrow & \psi_u=G*u,\\
-\Delta \psi_v=v& \Leftrightarrow & \psi_v=G*v.
\end{array}
\right.
\label{eq:coupling_poisson_decomposed}
\end{equation*}
Formally differentiating the diffusive energy \eqref{eq:def_energy_diff} with respect to $u$
(resp. $v$), a by now classical computation \cite{MR1842429, MR1964483} leads to the $\Delta u^m$
(resp. $\Delta v^m$) term  in \eqref{eq:PNP}. Similarly, differentiating
the external energy \eqref{eq:def_energy_external} with respect to
$u$ and $v$ classically gives rise to $\nabla\cdot(u\nabla U)$ and
$\nabla\cdot(v\nabla V)$ in \eqref{eq:PNP}. In order to
differentiate the coupling term we first use the formal integration
by parts \eqref{eq:formal_IBP} and then exploit the symmetry
$G(x-y)=G(y-x)$ to expand
\begin{align*}
\mathcal E_{\rm cpl}(u,v) & =\frac{1}{2}\iint_{\R^d\times\R^d}[u-v](x)G(x-y)[u-v](y)\, \mathrm{d}x\,\mathrm{d}y\\
& =\frac{1}{2}\int_{\R^d}
\big[u(G*u)+v(G*v)\big]\mathrm{d}x-\iint_{\R^d\times\R^d}u(x)G(x-y)v(y)\mathrm{d}x\,\mathrm{d}y.
\end{align*}
Differentiating with respect to $u$, it is well known
\cite{MR1964483} that the first integral gives the corresponding
$\nabla\cdot\left(u\nabla (G*u)\right)+0=\nabla\cdot\left(u\nabla
\psi_u\right)$ term. Rewriting the remaining cross term
$$
  -\iint_{\R^d\times\R^d} uGv\,\mathrm{d}x=-\int_{\R^d} u(G*v)\,\mathrm{d}x=-\int_{\R^d}
  u\psi_v\,\mathrm{d}x,
$$
and noting that $\psi_v$ is independent of $u$, it is again well
known that this term gives rise to $-\nabla\cdot\left(u\nabla
\psi_v\right)$. Summing up we obtain
$\nabla\cdot\left(u\nabla\left(\psi_u-\psi_v\right)\right)=\nabla\cdot\left(u\nabla\psi\right)$
 as in the first equation of \eqref{eq:PNP}. Similarly differentiating with respect to
 $v$ we obtain the $-\nabla\cdot\left(v\nabla\psi\right)$ term appearing in the second component.

Though very general notions of solutions related to Energy
Dissipation Equality (EDE) or Evolution Variational Inequality (EVI)
can be used for abstract gradient flows in metric spaces,
\cite{MR3050280, MR2401600}, we use the more direct framework, introducing some features later for the implementation of the flow-interchange method.
\begin{defi}
A pair $u, v: (0,\infty)\times\R^d\rightarrow\R^+$ is a global weak
solution if $u,v\in \mathcal{C}([0,\infty);\mathcal{P})$,
$u(t),v(t)\to u^0,v^0$ in $(\mathcal{P},d_W)$ as $t\searrow 0$,
$\nabla u^m$, $\nabla v^m$, $u\nabla U$, $v\nabla V$, $u\nabla\psi$
and $v\nabla\psi\in L^2(0,T;L^1(\R^d))$ for all $T>0$, and for any
fixed $\varphi\in \mathcal{C}^{\infty}_c(\R^d)$
\begin{align}
\frac{\mathrm{d}}{\mathrm{d}t}\int_{\R^d}
u(t,x)\varphi\,\mathrm{d}x&=-\int_{\R^d} \langle \nabla
u^m(t,x),\nabla\varphi\rangle\,\mathrm{d}x -\int_{\R^d}
u(t,x)\langle\nabla
U,\nabla\varphi\rangle\,\mathrm{d}x\nonumber\\
&\quad\;-\int_{\R^d}
u(t,x)\langle\nabla\psi(t,x),\nabla\varphi\rangle\,\mathrm{d}x,
 \label{eq:def_weak_formulation_u}
\end{align}
and
\begin{align}
\frac{\mathrm{d}}{\mathrm{d}t}\int_{\R^d}{v}(t,x)\varphi\,\mathrm{d}x&=-\int_{\R^d}\langle\nabla
v^m(t,x),\nabla\varphi\rangle\,\mathrm{d}x -\int_{\R^d}
v(t,x)\langle\nabla
V,\nabla\varphi\rangle\,\mathrm{d}x\nonumber\\
&\quad\;+\int_{\R^d}
v(t,x)\langle\nabla\psi(t,x),\nabla\varphi\rangle\,\mathrm{d}x
 \label{eq:def_weak_formulation_v}
\end{align}
hold in the sense of distributions $\mathcal{D}'(0,\infty)$ with $\psi(t,x)=G*[u-v](t,x)$ a.e. $(t,x)\in (0,\infty)\times\R^d$.
\label{def:weak_solutions}
\end{defi}
We observe that $L^2_{loc}([0,\infty);L^1(\R^d))$ could be replaced
by $L^1_{loc}((0,\infty)\times\R^d)$ in the above definition, which is
enough for all the integrals in \eqref{eq:def_weak_formulation_u}-\eqref{eq:def_weak_formulation_v}
to make sense. In any case the weak solutions constructed here would still enjoy strong regularity
in the end, and our choice of including the regularity in the definition of weak solutions is purely practical.

In this setting our main result is
\begin{theo}[Existence of solutions for $m>1$]
Fix $m>1$ and initial data $u^0,v^0$ as in \eqref{eq:initial_data}.
Then there exists a global weak solution $(u,v)$ with
\begin{equation}
\label{eq:regularity_weak_sol}
\begin{aligned}
& u,v\in
L^{\infty}\big(0,\infty; L^m(\R^d)\cap L^1(\R^d,(1+|x|^2)\mathrm{d}x)\big)\\
& \nabla\psi\in L^{\infty}(0,\infty;L^2(\R^d))\\
& u^{\frac{m}{2}}, v^{\frac{m}{2}}\in L^2(0, T; H^1(\R^d)) \\
& u,v\in L^{\infty}(0,T;L^p(\R^d)), \ \forall p\in [1,2d/(d+1)],
\end{aligned}
\end{equation}
%
%
for all $T>0$, and
\begin{equation}
\label{eq:energy_monotonicity} \mathcal E(u(t),v(t))\leq  \mathcal
E(u^0,v^0)\quad \text{for a.e. }t\geq 0.
\end{equation}
If we further assume that $u^0,v^0\in L^{p}(\R^d)$ for some
$p\in[1,\infty]$, then $\forall\tau\geq 0$,
\begin{equation}
\sup\limits_{t\in[0,\tau]}\left(\|u(t)\|_{L^p(\R^d)}+\|v(t)\|_{L^p(\R^d)}\right) \leq
Ce^{\lambda\tau}\left(\|u^0\|_{L^{p}(\R^d)}+\|v^0\|_{L^{p}(\R^d)}\right),
\label{eq:Lp_Linfty_estimate}
\end{equation}
with
\begin{equation}
\label{def of lambda}
 \lambda=\max\big\{\|\Delta U\|_{L^{\infty}(\R^d)},\|\Delta V\|_{L^{\infty}(\R^d)}\big\}.
\end{equation}
\label{theo:main_theo}
\end{theo}
In the case of linear diffusion we have similarly
\begin{theo}[Existence of solutions for $m=1$]
The conclusions of Theorem~\ref{theo:main_theo} hold for $m=1$
if we replace $u,v\in L^{\infty}(0,\infty;L^m(\R^d))$ in \eqref{eq:regularity_weak_sol} by \\  $u,v\in
L^{\infty}(0,\infty;L^1\log L^1(\R^d))$.
 \label{theo:main_theo_linear}
\end{theo}
\noindent We would like to stress again that estimate
\eqref{eq:Lp_Linfty_estimate} holds for $p=\infty$, and was not
known in the whole space as far as we can tell. Since we interpret
\eqref{eq:PNP} as a gradient flow one could expect energy
monotonicity $\mathcal E(t)\downarrow$. This would immediately
follow from \eqref{eq:energy_monotonicity} and uniqueness of
solutions. Unfortunately due to the lack of regularity and
displacement convexity we were not able to prove uniqueness within
the above class of weak solutions, and therefore we only retrieve an
energy upper bound.


It is worth mentioning that the gradient flow structure of the PNP
system and the above theorems are also valid in the bounded domain
case, with some mild assumptions on the boundary and minor
modifications of the proofs. Suppose $\Omega \subset \mathbb{R}^d$
is a smooth bounded and convex domain, and consider the physically
relevant boundary condition, that is, the no flux boundary condition
$$
  \frac{\partial{u}}{\partial{\nu}}=\frac{\partial{v}}{\partial\nu}=\frac{\partial\psi}{\partial\nu}=0
  \quad\mbox{on  } \partial\Omega, \quad
  \int_{\partial\Omega}\psi\,\mathrm{d}x=0,
$$
where $\nu$ is the unit outward normal on $\partial\Omega$. We
also assume the external potentials satisfy
$$
  \frac{\partial{U}}{\partial{\nu}}=\frac{\partial{V}}{\partial\nu}=0
  \quad\mbox{on  } \partial\Omega
$$
By \cite{MR1282720} we know that the electrostatic potential can be represented as
$$
  \psi(x)=\int_{\Omega}N(x, y)(u-v)(y)\,\mathrm{d}y, \qquad \forall x \in
  \Omega.
$$
Here the (singular) kernel $N(x,y)=N(y,x)$ serves as a counterpart of the Green's function $G(x-y)$ in
$\mathbb{R}^d$ for the Newton potential. Then we may argue in a similar way that the
PNP system formally possesses the gradient flow structure, and the
existence theorems can be proved in a similar but somewhat technically easier
manner.


\section{Study of the energy functionals}
\label{section:study_energy}
%
In this section we study various properties of the two relevant
energy functionals, namely the total free energy $\mathcal E$ and the
functional \eqref{def: energy functional in JKO scheme} used in the
JKO minimizing scheme. As already mentioned in the introduction, we
use the flow interchange technique to establish some improved
regularity for the minimizers, which will turn out to be crucial in
the next sections.

For further use we recall here a particular case of the celebrated Hardy-Littlewood-Sobolev (HLS) inequality:
\begin{lem}[Hardy-Littlewood-Sobolev, \cite{MR1817225,MR0290095}] In dimension $d\geq 3$ let $w\in L^p(\R^d)$ and $\Phi=G*w$. Then
\begin{enumerate}
 \item
 If $1<p<d/2$ there is a $C=C(p,d)$ such that
 \begin{equation}
 \|\Phi\|_{L^{\frac{dp}{d-2p}}(\R^d)}\leq C \|w\|_{L^p(\R^d)},
 \tag{HLS-1}
 \label{eq:HLS_G_p}
 \end{equation}
while if $p=1$ there is a $C=C(d)$ such that
 \begin{equation}
 \|\Phi\|_{L_w^{\frac{d}{d-2}}(\R^d)}\leq C \|w\|_{L^1(\R^d)}.
 \tag{HLS-2}
 \label{eq:HLS_G_1}
 \end{equation}
 \item
 If $1<p<d$ there is a $C= C_{p,d}$ such that
 \begin{equation}
 \|\nabla\Phi\|_{L^{\frac{dp}{d-p}}(\R^d)}\leq C\|w\|_{L^p(\R^d)}.
 \tag{HLS-3}
 \label{eq:HLS_nablaG_p}
 \end{equation}
\end{enumerate}
\label{lem:HLS}
\end{lem}
\noindent Since $|G(x)|=\frac{C}{|x|^{d-2}}$ and $|\nabla G(x)|=\frac{C}{|x|^{d-1}}$ this is a
 particular case of well known fractional integration results for the Riesz potential
 $I_{\alpha}f=\frac{1}{|.|^{d-\alpha}}*f$ with $\alpha=2,1$, and we refer to
 \cite{MR1817225,MR0290095} for details. Here $L^q_w(\R^d)$ denotes the weak-$L^q$
 space and coincides with the usual Lorentz space $L^{q,\infty}(\R^d)$.

As an immediate consequence we have the following integration by parts formula:
\begin{prop}
\label{prop:Poisson_IBP}
Let $d\geq 3$ and $w\in L^{2d/(d+2)}(\R^d)$. Then
$\Phi=(-\Delta)^{-1}w=G*w$ satisfies $\Phi\in L^{2d/(d-2)}(\R^d)$,
$\nabla \Phi\in L^2(\R^d)$, and
\begin{equation}
\int_{\R^d}|\nabla \Phi|^2\mathrm{d}x=\int_{\R^d}\Phi w\,\mathrm{d}x
=\iint_{\R^d\times\R^d} w(x)G(x-y)w(y)\mathrm{d}x\,\mathrm{d}y.
\label{eq:Poisson_IBP}
\end{equation}
\end{prop}
\noindent We shall use this later on with $w=u-v$ in order to
control $\psi=G*(u-v)$.
\begin{proof}
Taking $p=\frac{2d}{d+2}\in (1,d/2)$ in \eqref{eq:HLS_G_p}\eqref{eq:HLS_nablaG_p} we see that
 $\Phi\in L^{2d/(d-2)}(\R^d)$ and $\nabla\Phi\in L^2(\R^d)$. Since $(2d/(d+2))'=2d/(d-2)$ all the integrals in
\eqref{eq:Poisson_IBP} are absolutely convergent and the last equality holds
by Fubini's theorem. In order to retrieve the first equality we use approximation: if $w_n\in \mathcal{C}^{\infty}_c(\R^d)$
converges to $w$ in $L^{2d/(d+2)}(\R^d)$ then by the HLS lemma $\Phi_n\to\Phi$ in $L^{2d/(d-2)}(\R^d)$ and
$\nabla\Phi_n\to\nabla\Phi$ in $L^{2}(\R^d)$. Since \eqref{eq:Poisson_IBP} holds for smooth
$w_n\in \mathcal{C}^{\infty}_c$ with $-\Delta\Phi_n=w_n$ we conclude by letting $n\to\infty$.
\end{proof}
Back to our energy functional, we begin with a fairly standard type of result \cite{MR2383208,MR1617171}:
\begin{prop}[Energy lower bound]
Let $m\geq 1$ and $\mathcal{K}$ as in \eqref{eq:def_admissible_set_K}. The total free energy $\mathcal E$ is a proper functional on $\mathcal{K}\times\mathcal{K}$ and
\begin{equation*}
\inf\limits_{\mathcal{K}\times \mathcal{K}}\mathcal E(u,v)>-\infty.
\label{eq:lower_bound_energy}
\end{equation*}
Moreover we have in every sub-levelset $\{\mathcal E(u,v)\leq R\}$ that
\begin{enumerate}
\item[(i)] gradient control: $\|\nabla\psi\|_{L^2(\R^d)}\leq C$
 \item[(ii)] no concentration: if $m>1$ then
 $$
 \int_{\R^d}(u^m+v^m)\,\mathrm{d}x\leq C,$$
 while if $m=1$ then
 $$\int_{\R^d}(u|\log u|+v|\log v|)\,\mathrm{d}x\leq C. $$
 \item[(iii)] mass confinement: $\int_{\R^d} |x|^2(u+v)\,\mathrm{d}x\leq C,$
\end{enumerate}
for some $C>0$ depending on $R>0$, the confining potentials, and
$m$.
\label{prop:lower_bound_energy+equiintegrability}
\end{prop}
\begin{proof}
Choosing $u,v$ smooth and compactly supported it is clear that
$\mathcal E(u,v)<\infty$ so $\mathcal E$ is proper.

\noindent$\mathbf{m>1:}$ (i)-(ii)
immediately hold because each term in \eqref{eq:def_total_energy} is nonnegative.
(iii) then follows by $\int_{\R^d}(uU+vV)\,\mathrm{d}x\leq \mathcal E(u,v)$ together with
\eqref{eq:structural_potentials}.

\noindent$\mathbf{m=1:}$ if $\mathfrak{m}_2(\rho)=\int_{\R^d}|x|^2 \rho \mathrm{d}x$
denotes the second moment let us first recall \cite{MR1617171} the Carleman estimate
\begin{equation}
 \mathcal{H}(\rho) \geq
-\int_{\R^d} \rho (\log \rho)^-\,\mathrm{d}x\geq
-C(1+\mathfrak{m}_2(\rho))^{\alpha}, \rho\in \mathcal{P},
\label{eq:Carlemann}
\end{equation}
for some $C>0$ and $\alpha\in (0,1)$ depending on the dimension $d$
only. By \eqref{eq:structural_potentials} we have $\mathcal E_{\rm ext}(u,v)\geq
C_1(\mathfrak{m}_2(u)+\mathfrak{m}_2(v))$, whence
\begin{align*}
&C\left[\mathfrak{m}_2(u)+\mathfrak{m}_2(v)-(1+\mathfrak{m}_2(u))^{\alpha}-(1+\mathfrak{m}_2(v))^{\alpha}\right]\\
&\hspace{2cm}\leq
\mathcal{H}(u)+\mathcal{H}(v)+\int_{\R^d}(uU+vV)\,\mathrm{d}x\leq \mathcal  E(u,v)\leq
C.
\end{align*}
Hence the second moments are bounded as in (iii). Then (i) and (ii)
come immediately from (iii) and \eqref{eq:Carlemann}.
\end{proof}

For fixed $z_\ast=(u_\ast, v_\ast)\in
\mathcal{P}\times\mathcal{P}$, and given time step $h>0$ we set
\begin{equation}
\mathcal F_h(z):=\frac{1}{2h}d^2(z,z_*)+\mathcal E(z), \qquad\ z=(u,v)\in\mathcal{K}\times
\mathcal{K}.
\label{def: energy functional in JKO scheme}
\end{equation}
In order to define later a discrete sequence of approximate
solutions using the JKO minimizing scheme, we collect here some
properties of $\mathcal F_h$ and preliminary results.
\begin{prop}[Existence of minimizers]
Fix $h>0$, and $z_*=(u_*,v_*)\in\mathcal{P}\times\mathcal{P}$.
Then $\mathcal F_h$ admits a unique minimizer
$z=(u,v)\in\mathcal{K}\times\mathcal{K}$.
\label{prop:discrete_scheme_well_defined}
\end{prop}
\begin{proof}
By Proposition~\ref{prop:lower_bound_energy+equiintegrability},
$\mathcal F_h$ is bounded from below on $\mathcal{K}\times \mathcal{K}$,
hence there is a minimizing sequence $z_k=(u_k,v_k)$ satisfying
(i)-(iii) and $\{u_k,v_k\}_k$ are tight and uniformly
integrable. By the Dunford-Pettis Theorem one may extract a
subsequence such that
$$
u_k\rightharpoonup u\text{ and }v_k\rightharpoonup v\qquad \text{in
}L^1(\R^d),
$$
and standard truncation arguments together with the uniform bounds
on the second moments ensure that $u,v\in \mathcal{P}$. The weak
$L^1$ lower semi-continuity (l.s.c.) of the squared Wasserstein
distance, diffusive and potential energies are standard, in
particular $u,v\in \mathcal{K}$. We prove in the appendix,
Proposition~\ref{prop:Dirichlet_lsc} that the Dirichlet energy is
lower semicontinuous with respect to weak $L^1(\R^d)$ convergence.
Because $u_k-v_k\rightharpoonup u-v$ in $L^1(\R^d)$ we conclude here
that
%
%
$
\mathcal  E_{\rm cpl}(u,v)\leq \liminf \limits_{k\to\infty}\mathcal E_{\rm cpl}(u_k,v_k),
$
thus $u,v$ is a minimizer.
Finally, the uniqueness result comes from the fact that the admissible
set $\mathcal{K}\times \mathcal{K}$ is convex w.r.t. linear
interpolation $z_\theta=(1-\theta)z_0+\theta z_1$ and that the total
free energy is jointly strictly convex in $(u,v)$.
\end{proof}
We remark that the squared distance term left aside in \eqref{def: energy functional
in JKO scheme}, the same line of argument would readily give existence of a global
 minimizer of the total free energy $\mathcal E$, which would result in the end in a least energy
 stationary weak solution $\underline{u},\underline{v}$ to \eqref{eq:PNP}. Since our gradient
  flow system is driven by $\mathcal E$ one could expect long time convergence
  $u(t),v(t)\to \underline{u},\underline{v}$ when $t\to\infty$ together with some
   convergence rates. However, the lack of displacement convexity prevents here from
   applying standard techniques \cite{MR2401600,MR2053570,MR1842429} and this is beyond
    the scope of this paper. We refer to \cite{MR1777308,MR1866628} for related results on similar PNP models.

In the next section we shall derive the discrete Euler-Lagrange
equations satisfied by the minimizers, which requires integration by
parts as in \eqref{eq:Poisson_IBP}. However, at this stage the
minimizers only lie in $L^m(\R^d)$ if $m>1$ and $L^1\log L^1(\R^d)$
if $m=1$,
and this manipulation is not justified. The discrete Euler-Lagrange equations are necessary to pass to
the limit as the time step  $h \rightarrow 0$ and to thereby obtain a solution to the PNP system. The
remainder of this section is devoted to improving the regularity of the minimizers of the discrete functional.

The argument is based on the flow interchange technique of Matthes,
McCann, and Savar\'e, \cite{MR2581977}, as implemented by Blanchet
and Lauren\c{c}ot \cite{MR3040679}, as well as Lauren\c{c}ot and
Matioc \cite{MR3044141} and \cite{BCKKLL}. The idea of the flow interchange technique
is that a known gradient flow is sufficiently close, generally first
order close, to the one under study so that it may be used as an
approximation with controllable error. We need to use this method
twice, first to propagate the regularity of the minimizers and then
to establish some smoothness of their spatial gradients. The
characterization of gradient flow that is useful here is the so
called Evolution Variational Inequality (EVI) for a functional  $\mathcal F$.
Using the notation to follow, a flow  $(\tilde u(t))$  is a gradient
flow in the EVI sense, \cite{MR3050280,MR2401600}, provided that
\begin{equation}
\label{eq:GF_def}
\frac12\frac{\mathrm{d}}{\mathrm{d}t}d_W^2(\tilde u(t),w) + \mathcal F(\tilde
u(t)) \leq \mathcal F(w) \ \text{for all} \ w \in \mathcal{P}^{ac}(\R^d)\
\mbox{and a.e. }\ t>0.
\end{equation}
Displacement convexity and the other detailed requirements for \eqref{eq:GF_def} to hold are discussed in the references just cited. For our purposes we note that \eqref{eq:GF_def} is valid for
\begin{enumerate}
\item solutions of $\partial_t \tilde{u} = \Delta \tilde{u}$, the heat equation, with $\mathcal F = \mathcal{H}$, the Boltzmann entropy
\begin{equation}
\label{eq:3.1}
\mathcal{H}(\tilde{u}) = \int_{\R^d} \tilde{u} \log \tilde{u} \mathrm{d}x \ \ \mathrm{and}
\end{equation}
\item solutions of $\partial_t\tilde{u}=\Delta \left(\tilde{u}^p\right)$, the porous medium flow, with $\mathcal F = \mathcal{E}_p, 1<p<\infty,$ given by the functional
\begin{equation}
\label{eq:3.2}
\mathcal{E}_p(\tilde{u}) := \frac{1}{p-1}\int_{\R^d} \tilde{u}^p \mathrm{d}x.
\end{equation}
\end{enumerate}

\begin{prop}[Discrete propagation of $L^p$ estimates] Let $m\geq 1$, $\lambda$ as in \eqref{def of lambda},
and further assume that $u_*,v_*\in \mathcal{K}\cap L^p(\R^d)$ for
some $p\in(1,\infty)$. If $0<h<h_0(p)=\frac{1}{\lambda(p-1)}$ then
the minimizer $(u, v)$ from
Proposition~\ref{prop:discrete_scheme_well_defined} satisfies
 \begin{equation}
\|u\|_{L^p(\R^d)}^p+\|v\|^p_{L^p(\R^d)}\leq \frac{1}{1-\lambda(p-1)
h} \Big(\|u_*\|_{L^p(\R^d)}^p+\|v_*\|^p_{L^p(\R^d)}\Big).
\label{eq:discrete_Lp_propagation}
 \end{equation}
 \label{prop:discrete_Lp_propagation}
\end{prop}
In this first use of the flow interchange, we simply use the
solution of  \eqref{eq:PMEsystem} below as variations in the minimum
principle. Note that at this point the time step $h$ must be taken
small in terms of $p$ for the minimizing problem to ``see'' the
estimate. As a consequence there is no hope to retrieve an
$L^{\infty}(\R^d)$ estimate at the discrete level for fixed $h$
directly from the limit $p\to\infty$ in
\eqref{eq:discrete_Lp_propagation}, since $h<h_0(p)$ would require
$h\to 0$. However, $u,v$ will be retrieved as some limit when $h\to
0$, so one can actually take $p$ arbitrarily large and the weak
solutions will ultimately satisfy such an $L^{\infty}$ estimate. See
the proof of Theorem~\ref{theo:main_theo} at the end of
Section~\ref{section:CV_to_weak_solution} for details.
\begin{proof}
For fixed $p\in(1,\infty)$ and $u_*,v_*\in \mathcal{K}\cap L^p$
consider the auxiliary Porous Media flows $\tilde u(t),\tilde{v}(t)$ defined by
\begin{equation}
\label{eq:PMEsystem}
\begin{aligned}
&\partial_t\tilde{u}=\Delta \left(\tilde{u}^p\right)\text{ in }(0,\infty)\times\R^d, & \tilde{u}|_{t=0}=u \text{ in }\R^d,\\
&\partial_t\tilde{v}=\Delta \left(\tilde{v}^p\right) \text{ in }(0,\infty)\times\R^d, & \tilde{v}|_{t=0}=v \text{ in }\R^d.
\end{aligned}
\end{equation}
By standard results for the PME \cite{MR2286292} we know that
(i) these Cauchy problems are well posed and $\tilde{u},\tilde{v}\in\mathcal{C}\big([0,\infty);L^1(\R^d)\big)$ remain probability measures,
(ii) by $L^1-L^{\infty}$ smoothing $\tilde{u}(t),\tilde{v}(t)\in L^{\infty}(\R^d)$ for all $t>0$,
and (iii) the second moments remain finite.
 As a consequence $\tilde{u}(t),\tilde{v}(t)\in \mathcal{K}$ are admissible for any $t>0$, and by
 Proposition \ref{prop:Poisson_IBP} it will be no issue to integrate by parts in the coupling term.

\noindent\textbf{Step 1: dissipation of the internal energy.} We
first claim that
\begin{equation}
t>0:\qquad \frac{\mathrm{d}}{\mathrm{d}t}\mathcal E_{\rm diff}(\tilde{u},\tilde{v})\leq 0,
\label{eq:dissipation_Ediff}
\end{equation}
and we distinguish cases, $m=1$ being the most involved.

\noindent \textit{$m>1:$} By usual properties \cite{MR2286292} of
the PME all the $L^q(\R^d)$ norms are non-increasing along the PME
flow, in particular
$\mathcal E_{\rm diff}(\tilde{u},\tilde{v})=\frac{1}{m-1}\left(\|\tilde{u}\|^m_{L^m(\R^d)}+\|\tilde{v}\|^m_{L^m(\R^d)}\right)$
is nonincreasing in time.

\noindent \textit{$m=1:$} Assume first that $u$ is smooth and
positive. Then by standard properties of the PME flow so is $\tilde{u}(t)$ for later times and, therefore, we
get
\begin{equation}
\frac{\mathrm{d}}{\mathrm{d}t}\Big(\int_{\R^d}\tilde{u}\log\tilde{u}\,\mathrm{d}x\Big)
=\int_{\R^d}(\log\tilde{u})\Delta\tilde{u}^p\,\mathrm{d}x=-\frac{4}{p}\int_{\R^d}|\nabla\tilde{u}^{p/2}|^2\mathrm{d}x
\leq 0
\label{eq:dissipation_Boltzmann_along_PME}
\end{equation}
for $t>0$. Thus $t\mapsto \mathcal{H}(\tilde u(t))$ is nonincreasing.

If $u$ is not smooth and positive we first regularize it by running the heat equation for small times
$$
u_k=\Gamma_{1/k}*u\underset{k\to\infty}{\rightarrow}u,
$$
where $\Gamma_s$ is the usual heat kernel at time $s$. Since the heat equation is the $\mathcal{H}$-gradient flow we have in particular $\mathcal{H}(u_k)\leq \mathcal{H}(u)$, and of course $u_k$ is positive and smooth.
Denoting by $\tilde{u}_k(t)$ the solution of the corresponding PME-flow $\partial_t\tilde{u}_k=\Delta\tilde{u}^p_k$ starting from $u_k$, then the previous computation \eqref{eq:dissipation_Boltzmann_along_PME} shows that
$$
\forall \,t>0:\qquad \mathcal{H}(\tilde u_k(t))\leq \mathcal{H}(\tilde u_k(0))\leq \mathcal{H}(u).
$$
Since $u_k\to u$ in $L^1(\R^d)$, we get by standard $L^1$ contractivity of the PME that $\tilde{u}_k(t)\to\tilde{u}(t)$ in $L^1(\R^d)$ uniformly in $t\geq 0$ when $k\to\infty$, in particular weakly in $L^1(\R^d)$. By lower semi-continuity of $\mathcal{H}$ with respect to weak $L^1$ convergence we conclude that
$$
\forall \,t>0:\qquad \mathcal{H}(\tilde u(t))\leq \liminf\limits_{k\to\infty}\mathcal{H}(\tilde u_k(t))\leq \mathcal{H}(u).
$$
Finally, by uniqueness of solutions to the PME flow with
$\tilde{u}(0)=u$ we conclude that $t\mapsto\mathcal{H}(\tilde u(t))$
is monotone nonincreasing, and similarly arguing for $v$ entails
\eqref{eq:dissipation_Ediff}  as claimed.

\noindent\textbf{Step 2: the remaining terms.} Arguing by
approximation \cite{MR2286292} the potential energy is easily
controlled for $t>0$ as
\begin{align}
\frac{\mathrm{d}}{\mathrm{d}t} \mathcal E_{\rm
ext}(\tilde{u},\tilde{v})&=\frac{\mathrm{d}}{\mathrm{d}t}\int_{\R^d}
\big(\tilde{u}U+\tilde{v}V\big)\,\mathrm{d}x
= \int_{\R^d}\big[(\Delta\tilde{u}^p)U+(\Delta\tilde{v}^p)V\big]\,\mathrm{d}x\nonumber\\
  & =\int_{\R^d}\big[\tilde{u}^p\Delta U+\tilde{v}^p\Delta V\big]\,\mathrm{d}x
  \leq\lambda\int_{\R^d}(\tilde{u}^p+\tilde{v}^p)\,\mathrm{d}x.
  \label{eq:dissipation_Eext}
\end{align}
For the coupling term, let
$\tilde{\psi}(t)=G*(\tilde{u}-\tilde{v})(t)$ and observe that for
$t>0$ we have
$\partial_t\tilde{\psi}=\partial_t[(-\Delta)^{-1}(\tilde{u}-\tilde{v})]=(-\Delta)^{-1}[\partial_t(\tilde{u}-\tilde{v})]=-(\tilde{u}^p-\tilde{v}^p)$.
Since $\tilde{u}(t),\tilde{v}(t)\in L^{\infty}(\R^d)$ for $t>0$ we can legitimately
integrate by parts
\begin{align}
\frac{\mathrm{d}}{\mathrm{d}t} \mathcal E_{\rm cpl}(\tilde{u},\tilde{v})& =\frac{\mathrm{d}}{\mathrm{d}t}\left(\frac{1}{2}\int_{\R^d}|\nabla\tilde{\psi}|^2\,\mathrm{d}x\right) \nonumber\\
  & =\int_{\R^d} (-\Delta\tilde{\psi})\partial_t\tilde{\psi}\,\mathrm{d}x
  =-\int_{\R^d} (\tilde{u}-\tilde{v}).(\tilde{u}^p-\tilde{v}^p)\,\mathrm{d}x
  \leq 0.
  \label{eq:dissipation_Ecpl}
\end{align}
Note that, due to $\|\nabla \tilde{\psi}\|_{L^2(\R^d)}=\|\nabla (-\Delta)^{-1}(\tilde{u}-\tilde{v})\|_{L^2(\R^d)}\approx \|\tilde{u}-\tilde{v}\|_{H^{-1}(\R^d)}$, this is the well-known $H^{-1}$ contraction property of the PME flow, see  \cite{MR2286292}.

As for the Wasserstein term, note that $\tilde{u},\tilde{v}$ are
 respective gradient flows of the functional $\mathcal{E}_p$, \eqref{eq:3.2}, so from \eqref{eq:GF_def},
\begin{align}
&\frac{1}{2h}\frac{\mathrm{d}}{\mathrm{d}t}\big[ d_W(\tilde{u},u_*)^2+ d_W(\tilde{v},v_*)^2\big]\nonumber\\
&\quad\leq\frac{1}{(p-1)h}\int_{\R^d}\left(u_*^p-\tilde{u}^p\right)\,\mathrm{d}x
+\frac{1}{(p-1)h}\int_{\R^d}\left(v_*^p-\tilde{v}^p\right)\,\mathrm{d}x.
\label{eq:dissipation_W2}
\end{align}

\noindent\textbf{Step 3: dissipation inequality.} Gathering
\eqref{eq:dissipation_Ediff}, \eqref{eq:dissipation_Eext},
\eqref{eq:dissipation_Ecpl}, and \eqref{eq:dissipation_W2}, we get
the total dissipation inequality
\begin{align*}
\mathcal D(t):=\frac{\mathrm{d}}{\mathrm{d}t}\mathcal F_h(\tilde{u},\tilde{v})
&\leq\frac{1}{(p-1)h}\int_{\R^d} (u_*^p+v_*^p)\,\mathrm{d}x-\frac{1}{(p-1)h}\int_{\R^d}(\tilde{u}^p+\tilde{v}^p)\,\mathrm{d}x\nonumber\\
  &\qquad+\lambda\int_{\R^d}(\tilde{u}^p+\tilde{v}^p)\,\mathrm{d}x=: \mathcal A(t)
\end{align*}
for small $t>0$. Because $\big(\tilde{u}(0),\tilde{v}(0)\big)=(u,v)$
is a minimizer we must have $D(t)\geq 0$ at least for a time
sequence $t_n\searrow 0$, otherwise $(\tilde{u}(t),\tilde{v}(t))$ would be
a strictly better competitor for small $t>0$. If
$0<h<h_0=\frac{1}{\lambda(p-1)}$ we have $1-\lambda(p-1)h>0$ and
$\mathcal A(t_n)\geq \mathcal D(t_n)\geq 0$ can be rearranged as
\begin{equation}
\label{eq:3.3}
\int_{\R^d}\big[\tilde{u}^p(t_n)+\tilde{v}^p(t_n)\big]\,\mathrm{d}x\leq
\frac{1}{1-\lambda(p-1)h}\int_{\R^d}(u^p_*+v_*^p)\,\mathrm{d}x.
\end{equation}
Our statement follows by finally letting $t_n\searrow 0$ in
\eqref{eq:3.3}, recalling that
$\big(\tilde{u}(t),\tilde{v}(t)\big)\to (u, v)$ in $L^1(\R^d)$
 when $t\to 0$.
\end{proof}
We shall also need a further regularity result for the gradient of  $(u,v)$. The use of the flow interchange in this estimate is very similar to its use in \cite{MR3040679} for the critical parabolic-parabolic Keller-Segel model.

\begin{prop}[Discrete gradient estimate]
For $m\geq 1$, $d\geq 3$, and any $h>0$, fix $z_*=(u_*,v_*)\in
\mathcal{P}\times\mathcal{P}$ and let
$z=(u,v)\in\mathcal{K}\times\mathcal{K}$ be the unique minimizer
from Proposition \ref{prop:discrete_scheme_well_defined}. Then
\begin{align}
&\big\|\nabla(u^{m/2})\big\|^2_{L^2(\R^d)}+\big\|\nabla(v^{m/2})\big\|^2_{L^2(\R^d)}\nonumber\\
&\hspace{1cm}\leq
C\Big[1+\frac{\mathcal{H}(u_*)-\mathcal{H}(u)}{h}+\frac{\mathcal{H}(v_*)-\mathcal{H}(v)}{h}\Big]
\label{eq:discrete_gradient_np^m_estimate}
\end{align}
for some $C>0$ independent of $h>0$ and $z_*$.
\label{prop:discrete_gradient_np_estimate}
\end{prop}
%
%
\begin{proof}
We use a second flow interchange with $\tilde u(t),\tilde{v}(t)$
now defined by
\begin{equation}
\partial_t\tilde{u}-\Delta \tilde{u}=0 \text{ in }(0,\infty)\times\R^d,\qquad \tilde{u}|_{t=0}=u \text{ in }\R^d
\label{eq:heat_u}
\end{equation}
and
\begin{equation}
\partial_t\tilde{v}-\Delta \tilde{v}=0 \text{ in }(0,\infty)\times\R^d,\qquad \tilde{v}|_{t=0}=v \text{ in }\R^d.
\label{eq:heat_v}
\end{equation}

\noindent\textbf{Step 1: dissipation inequality.} We first note that
classical properties of the heat equation and
$\tilde{u}(0),\tilde{v}(0)\in\mathcal{K}$ guarantee
$\tilde{u}(t),\tilde{v}(t)\in \mathcal{K}$ for all $t>0$. Let
$\tilde{\psi}:=G*(\tilde{u}-\tilde{v})$. Then it is easy to check
that $\partial_t\tilde{\psi}=\Delta\tilde{\psi}$ as well. Since the
pair $(u,v)$ is a minimizer and has finite energy we have in
particular $\nabla\tilde{\psi}(0)=\nabla\psi\in L^2(\R^d)$, whence by standard properties of the heat equation
$\nabla\tilde{\psi}(t)\in L^2(\R^d)$ and
$$
\frac{\mathrm{d}}{\mathrm{d}t}\big\|\nabla\tilde{\psi}(t)\big\|^2_{L^2(\R^d)}\leq
0
$$
for all $t> 0$. Since $\tilde{u},\tilde{v}$ are positive and smooth
for all $t>0$ we may differentiate
and integrate by parts as

\noindent $\mathbf{m>1}:$
\begin{align}
&\frac{\mathrm{d}}{\mathrm{d}t}\mathcal E(\tilde{u},\tilde{v})\nonumber\\
&= \frac{m}{m-1}\int_{\R^d}\big(\tilde{u}^{m-1}\Delta\tilde{u}
+\tilde{v}^{m-1}\Delta\tilde{v}\big)\,\mathrm{d}x+\int_{\R^d}
\big(U\Delta\tilde{u}+V\Delta\tilde{v}\big)\,\mathrm{d}x\nonumber\\
&\qquad+\frac{\mathrm d}{\mathrm d t}\Big(\frac{1}{2}\big\|\nabla\tilde{\psi}(t)\big\|^2_{L^2(\R^d)}\Big)\nonumber\\
&\leq
-\frac{4}{m}\int_{\R^d}\Big(|\nabla\tilde{u}^{\frac{m}{2}}|^2+|\nabla\tilde{v}^{\frac{m}{2}}|^2\Big)\mathrm{d}x
  +\int_{\R^d}\big(\tilde{u}\Delta{U}+\tilde{v}\Delta{V}\big)\,\mathrm{d}x\nonumber\\
&\leq -\frac{4}{m}\int_{\R^d}\Big(|\nabla \tilde{u}^{\frac{m}{2}}|^2
+|\nabla\tilde{v}^{\frac{m}{2}}|^2\Big)\,\mathrm{d}x+\|\Delta{U}\|_{L^{\infty}(\R^d)}+\|\Delta{V}\|_{L^{\infty}(\R^d)}
   \label{eq:energy_dissipation_flow_interchange_m>1}
 \end{align}
 and similarly\\
\noindent $\mathbf{m=1}:$
 \begin{equation}
 \frac{\mathrm{d}}{\mathrm{d}t}\mathcal E(\tilde{u},\tilde{v})
\leq -4\int_{\R^d}\Big(|\nabla\tilde{u}^{\frac{1}{2}}|^2+|\nabla
\tilde{v}^{\frac{1}{2}}|^2\Big)\,\mathrm{d}x+\|\Delta{U}\|_{L^{\infty}(\R^d)}+\|\Delta{V}\|_{L^{\infty}(\R^d)}.
\label{eq:energy_dissipation_flow_interchange_m=1}
\end{equation}
Here \eqref{eq:structural_potentials} is used in the third
inequality of both
\eqref{eq:energy_dissipation_flow_interchange_m>1} and
\eqref{eq:energy_dissipation_flow_interchange_m=1}.

For $m>1$, the above term
$\frac{4}{m}\int_{\R^d}|\nabla\tilde{\rho}^{\frac{m}{2}}|^2\mathrm{d}x$
in \eqref{eq:energy_dissipation_flow_interchange_m>1} (dissipation
of $\mathcal{E}_m(\rho)=\frac{1}{m-1}\int_{\R^d}\rho^m\,\mathrm{d}x$
along $\partial_t\rho=\Delta\rho$) corresponds to the usual Fisher
information
$-4\int_{\R^d}|\nabla\rho^{\frac12}|^2\mathrm{d}x=\frac{\mathrm{d}}{\mathrm{d}t}\mathcal{H}(\rho)$
in the linear diffusion case
\eqref{eq:energy_dissipation_flow_interchange_m=1}, and enjoys a
formal continuity when $m\searrow 1$. Comparing with
\eqref{eq:dissipation_Boltzmann_along_PME} for $p=m$ we also see
that the dissipation of $\mathcal{H}$ along the $\mathcal{E}_m$-flow
equals the dissipation of $\mathcal{E}_m$ along the
$\mathcal{H}$-flow, which is in fact the cornerstone of this flow
interchange technique.

%
Let
\begin{equation*}
\mathcal D(t):=\frac{4}{m}\left(\big\|\nabla\tilde{u}^{\frac{m}{2}}(t)\big\|_{L^2(\R^d)}^2
+\big\|\nabla \tilde{v}^{\frac{m}{2}}(t)\big\|_{L^2(\R^d)}^2\right).
\label{eq:def_D}
\end{equation*}
Integrating \eqref{eq:energy_dissipation_flow_interchange_m>1} or
\eqref{eq:energy_dissipation_flow_interchange_m=1} from $0$ to $t>0$
we get in both cases
\begin{equation}
\mathcal E(\tilde{u}(t),\tilde{v}(t))-\mathcal E(u_*,v_*) \leq
2\lambda{t}-\int_0^t\mathcal D(s)\mathrm{d}s,
\label{eq:estimate_energy_interchange}
\end{equation}
with $\lambda$ defined in \eqref{def of lambda}. Because
\eqref{eq:heat_u}-\eqref{eq:heat_v} are respective
$\mathcal{H}$-gradient flows, we again appeal to \eqref{eq:GF_def}
to obtain
\begin{align*}
\frac{1}{2}\frac{\mathrm{d}}{\mathrm{d}t}d^2(\tilde{z}(t),z_*)& =
\frac{1}{2}\frac{\mathrm{d}}{\mathrm{d}t}\big[d_W^2(\tilde{u}(t),u_*)+d_W^2(\tilde{v}(t),v_*)\big]\\
&\leq \mathcal{H}(u_*)-\mathcal{H}(\tilde{u}(t))+
\mathcal{H}(v_*)-\mathcal{H}(\tilde{v}(t)).
\end{align*}
Integrating again and using the monotonicity of $s\searrow
\mathcal{H}(\tilde{u}(s))$ and $s\searrow \mathcal{H}(\tilde{v}(s))$
along the flow with $\tilde{z}(0)=z$ gives
\begin{equation*}
\frac{1}{2h}\big[d^2(\tilde{z}(t),z_*)-d^2(z,z_*)\big]\nonumber\leq\frac{t}{h}\big[\mathcal{H}(u_*)-\mathcal{H}(\tilde{u}(t))+
\mathcal{H}(v_*)-\mathcal{H}(\tilde{v}(t))\big].
\label{eq:estimate_distance_interchange}
\end{equation*}
Since $z$ is a minimizer we have by \eqref{def: energy functional in
JKO scheme} and \eqref{eq:estimate_energy_interchange} that for
small $t>0$
\begin{align*}
0& \leq  \mathcal F_h(\tilde{z}(t)) -\mathcal F_h(z)\\
 &  \leq\frac{t}{h}\big[\mathcal{H}(u_*)-\mathcal{H}(\tilde{u}(t))
 +\mathcal{H}(v_*)-\mathcal{H}(\tilde{v}(t))\big]+2\lambda{t}-\int_0^tD(s)\,\mathrm{d}s,
\end{align*}
which we reformulate as
\begin{equation}
 \frac{1}{t}\int_0^t\mathcal D(s)\,\mathrm{d}s\leq 2\lambda+\frac{\mathcal{H}(u_*)-\mathcal{H}(\tilde{u}(t))}{h}
 +\frac{\mathcal{H}(v_*)-\mathcal{H}(\tilde{v}(t))}{h}.
 \label{eq:estimate_D(t)_interchange}
\end{equation}
\textbf{Step 2: the limit $t\to 0$.} If
$$
  D_1(t,x):=\frac{1}{t}\int_0^t\tilde{u}^{\frac{m}{2}}(s,x)\mathrm{d}s,\qquad  D_2(t,x):=\frac{1}{t}\int_0^t\tilde{v}^{\frac{m}{2}}(s,x)\mathrm{d}s,
$$
we first note that
$\tilde{u},\tilde{v}\in\mathcal{C}([0,\infty);L^m(\R^d))$ as
solutions of the heat equation with initial data in $L^m(\R^d)$, so
that $D_1,D_2\in\mathcal{C}([0,\infty);L^2(\R^d))$. As a consequence
$D_1(t)\to D_1(0)=u^{\frac{m}{2}}$ and $D_2(t)\to
D_2(0)=v^{\frac{m}{2}}$ in $L^2(\R^d)$ when $t\downarrow 0$. By
\eqref{eq:estimate_D(t)_interchange} we find that $\nabla D_1(t)$
and $\nabla D_2(t)$ are bounded in $L^2(\R^d)$ and converge at least in
$\mathcal{D}'(\R^d)$ to $\nabla\big(u^{\frac{m}{2}}\big)$ and
$\nabla\big(v^{\frac{m}{2}}\big)$ when $t\to 0$. Consequently,
$\nabla\big(u^{\frac{m}{2}}\big),\nabla\big(v^{\frac{m}{2}}\big)\in
L^2(\R^d)$ and our statement follows from
\begin{align*}
\big\|\nabla( u^{\frac{m}{2}})\big\|_{L^2(\R^d)}^2&+\big\|\nabla(v^{\frac{m}{2}})\big\|_{L^2(\R^d)}^2 \nonumber\\
&\leq \liminf\limits_{t\to 0}\big\|\nabla D_1(t)\big\|_{L^2(\R^d)}^2+\liminf\limits_{t\to 0}\big\|\nabla D_2(t)\big\|_{L^2(\R^d)}^2\\
&\leq \frac{m}{4}\liminf\limits_{t\to 0}\frac{1}{t}\int_0^t \mathcal D(s)ds\\
    &   \leq C\left(1+\frac{\mathcal{H}(u_*)-\mathcal{H}(u)}{h}+ \frac{\mathcal{H}(v_*)-\mathcal{H}(v)}{h}\right).
\end{align*}
The last inequality comes from the limit $t\searrow 0$
in \eqref{eq:estimate_D(t)_interchange} with strong convergence
$(\tilde{u}(t),\tilde{v}(t))\to (u,v)$, e.g. in $L^1\cap L^{r_0}(\R^d)$, and suitable continuity of $\rho\mapsto
\mathcal{H}(\rho)$.
\end{proof}
\section{Minimizing scheme and discrete estimates}
\label{section:minimizing_scheme}
In this section, we shall construct a family of time-discrete
approximate solutions using the JKO method, also known as the variational minimizing movement
scheme. {\it A priori} estimates for the set of discrete solutions
are necessary to allow us to deduce the existence of a time-continuous
limit curve.

 \label{section:discrete_scheme} Fix an
initial datum $z^0=(u^0,v^0)$ as in \eqref{eq:initial_data} and some
time step $h>0$. Setting $z_h^{(0)}=z^0$, Proposition
\ref{prop:discrete_scheme_well_defined} allows us to define a
sequence $z_h^{(n)}=(u_h^{(n)},v_h^{(n)}) \in \mathcal{K}\times\mathcal{K}$ recursively as
$$
z_h^{(n+1)}:= \text{ the unique minimizer } z \text{ of }\mathcal F_h\text{
with }z_*=z_h^{(n)}=\big(u_h^{(n)},v_h^{(n)}\big)
$$
and a corresponding piecewise-constant interpolation
$t\in[0,\infty)\mapsto z_h(t)$ as
\begin{equation*}
\label{eq:4.1}
z_h(t)=z_h^{(n)} \quad {\rm for} \quad nh \leq t < (n+1)h .
\end{equation*}
The rest of this section is devoted to collecting the suitable {\it a priori}
estimates on $z_h$ suitable to pass to the limit in $h\searrow 0$.

It is now standard to get the discrete energy monotonicity as in
\cite{MR1617171} that
$$
\forall\,n\geq 0:\qquad \mathcal E\big(z_h^{(n+1)}\big)\leq \mathcal E\big(z_h^{(n)}\big)
$$
inasmuch as  $z_h^{(n)}$  is a competitor in the search for
$z_h^{(n+1)}$.  At the continuous level this reads as
\begin{equation*}
 \mathcal E(z_h(t_2))\leq \mathcal E(z_h(t_1))\leq \mathcal E(z^0) \quad {\rm for \; all} \quad 0\leq t_1 \leq t_2.
 \label{eq:interpolated_energy_monotonicity}
\end{equation*}
\begin{prop}
The total square distance and approximate H\"older estimates
\begin{equation}
 \frac{1}{2h}\sum\limits_{n\geq 0}d^2\big(z_h^{(n)},z_h^{(n+1)}\big)\leq \mathcal E(z^0)-\inf\limits_{\mathcal{K}\times \mathcal{K}}
 \mathcal E,
 \label{eq:discrete_square_distance_estimate}
\end{equation}
\begin{equation}
\forall\,0\leq t_1\leq t_2,\qquad d\big(z_h(t_1),z_h(t_2)\big)\leq
C\left|t_2-t_1+h\right|^{\frac{1}{2}},
\label{eq:interpolated_approximate_Holder_estimate}
\end{equation}
hold for some $C>0$ independent of $h>0$.
\end{prop}
\begin{proof}
Note that since $u^0,v^0\in L^1(\R^d)\cap L^{2d/(d+1)}(\R^d)$ we
have in particular $u^0-v^0\in L^{2d/(d+2)}(\R^d)$. By
Proposition~\ref{prop:Poisson_IBP} we have $\nabla\psi^0=\nabla
G*(u^0-v^0)\in L^2(\R^d)$ and $u^0,v^0$ has therefore finite energy.
We also recall from Proposition
\ref{prop:lower_bound_energy+equiintegrability} that
$\inf\limits_{\mathcal{K}\times \mathcal{K}}
 \mathcal E>-\infty$, hence the right-hand side in \eqref{eq:discrete_square_distance_estimate} is finite (and of course independent of $h>0$). The rest of the argument is by now very classical and we refer to \cite{MR1617171}.
\end{proof}
\begin{prop}
The piecewise constant interpolation satisfies
\begin{equation}
\label{eq:interpolated_Lm_estimate}
\begin{aligned}
m>1:&\sup\limits_{t\geq 0}\int_{\R^d}\big(u_h^m(t)+v_h^m(t)\big)\,\mathrm{d}x  \leq C\\
m=1:&\sup\limits_{t\geq 0}\int_{\R^d}\big(u_h(t)|\log
u_h(t)|+v_h(t)|\log v_h(t)|\big)\,\mathrm{d}x  \leq C
\end{aligned}
\end{equation}
and
\begin{equation}
\label{eq:interpolated_moment_estimate}
\sup\limits_{t\geq 0}\int_{\R^d}|x|^2\big(u_h(t)+v_h(t)\big)\,\mathrm{d}x\leq
C
\end{equation}
uniformly in $h>0$.
\label{prop:interpolated_Lm_2nd_moment_estimate}
\end{prop}
\begin{proof}
By energy monotonicity we have $\sup\limits_{n\geq
0}\mathcal E\big(u_h^{(n)},v_h^{(n)}\big)\leq \mathcal E(u^0,v^0)<\infty$, which by
Proposition~\ref{prop:lower_bound_energy+equiintegrability} bounds
the internal energy and the second moments uniformly in $h,n$ for
the discrete sequence. This property extends to the interpolation $u_h(t),v_h(t)$.
\end{proof}
In addition to the uniform control in Proposition
\ref{prop:interpolated_Lm_2nd_moment_estimate}, we also have
\begin{prop}[Continuous $L^p$ estimate] In addition to \eqref{eq:initial_data} assume that the
initial data $u^0,v^0\in L^p(\R^d)$ for some $p\in(1,\infty)$, and let $\lambda$ as in \eqref{def of lambda}.
 Then for $h<h_0(p)=\frac{1}{\lambda(p-1)}$ sufficiently small we have
$$
\forall t\geq 0:\quad
\big\|u_h(t)\big\|_{L^p(\R^d)}+\big\|v_h(t)\big\|_{L^p(\R^d)}\leq
Ce^{\lambda t}\Big(\|u^0\|_{L^p(\R^d)}+\|v^0\|_{L^p(\R^d)}\Big)
$$
for some $C>0$ independent of $t$, $p$, $h$, and the initial data.
\label{prop:continuous_Lp_propagation}
\end{prop}
\begin{proof} Fix any $t>0$, let $k=\lfloor t/h\rfloor$, and recall that $u_h(t)=u_h^{(k)}$. By
induction we immediately get from
Proposition~\ref{prop:discrete_Lp_propagation}
$$
 \|u_h(t)\|_{L^p(\R^d)}^p+\|v_h(t)\|^p_{L^p(\R^d)}  \leq
 \Big(\frac{1}{1-\lambda(p-1)h}\Big)^{\lfloor t/h\rfloor}\big(\|u^0\|^p_{L^p(\R^d)}+\|v^0\|^p_{L^p(\R^d)}\big).
 $$
For small $h>0$ this easily gives
$$
\|u_h(t)\|_{L^p(\R^d)}+\|v_h(t)\|_{L^p(\R^d)}\leq
Ce^{\lambda\frac{p-1}{p}t}\big(\|u^0\|_{L^p(\R^d)}+\|v^0\|_{L^p(\R^d)}\big)
$$
for some universal $C>0$. Since $e^{\lambda\frac{p-1}{p}t}\leq
e^{\lambda t}$ the proof is complete.
\end{proof}
\begin{prop}[Approximate Euler-Lagrange equations]
Fix $m\geq 1$. Let $\nabla q^{(n)}$ and $\nabla
r^{(n)}$ be the optimal transport maps
$$
u_h^{(n+1)}=\big(\nabla
{q^{(n)}}\big)_{\#}u_h^{(n)}\quad\text{and}\quad
v_h^{(n+1)}=\big(\nabla {r^{(n)}}\big)_{\#}v_h^{(n)}
$$
in Brenier's Theorem~\ref{theo:Brenier}, and
$\psi_h^{(n)}=G*\left(u_h^{(n)}-v_h^{(n)}\right)$.
Then for any
vector-field $\zeta\in \mathcal{C}^{\infty}_c(\R^d;\R^d)$, we have that
\begin{align}
&\frac{1}{h}\int_{\R^d}\langle\nabla
q^{(n)}-\mathrm{Id},\zeta\circ\nabla
q^{(n)}\rangle u^{(n)}_h\mathrm{d}x\nonumber\\
&=\int_{\R^d}\big(u^{(n+1)}_h\big)^m\operatorname{div}\zeta\,\mathrm{d}x
-\int_{\R^d}u^{(n+1)}_h\langle\nabla{U},\zeta\rangle\,\mathrm{d}x
-\int_{\R^d}u_h^{(n+1)}\langle\nabla\psi_h^{(n+1)},\zeta\rangle\,\mathrm{d}x,
\label{eq:discrete_Euler-Lagrange_u}
\end{align}
and
\begin{align}
&\frac{1}{h}\int_{\R^d}\langle\nabla r^{(n)}-\mathrm{Id},\zeta\circ\nabla r^{(n)}\rangle v^{(n)}_h\mathrm{d}x \nonumber\\
&=\int_{\R^d}\big(v^{(n+1)}_h\big)^m \operatorname{div}\zeta\,\mathrm{d}x
-\int_{\R^d}v^{(n+1)}_h\langle\nabla{V},\zeta\rangle\,\mathrm{d}x
+\int_{\R^d}v_h^{(n+1)}\langle\nabla\psi_h^{(n+1)},\zeta\rangle\,\mathrm{d}x.
\label{eq:discrete_Euler-Lagrange_v}
\end{align}
\label{prop:weak_Euler-Lagrange}
\end{prop}
\begin{proof}
In order to simplify notations, we write below
$u_*=u^{(n)}_h$, $u=u_h^{(n+1)}$, $v_*=v^{(n)}_h$, $v=v_h^{(n+1)}$, and
$\psi=\psi_h^{(n+1)}=G*[u_h^{(n+1)}-v_h^{(n+1)}]$. Fix an arbitrary vector-field $\zeta\in
\mathcal{C}^{\infty}_c(\R^d,\R^d)$. For $\eps\in[-\delta,\delta]$,
let $\Phi_{\eps}(x)$ be the associated $\eps$-flow (i-e
$d\Phi_{\eps}/d\eps=\zeta(\Phi_{\eps})$ and $\Phi_0=\mathrm{Id}$),
and let us consider the perturbation (of domain)
$$
u_{\eps}:={(\Phi_{\eps})}_{\#}u,\qquad z_{\eps}:=(u_{\eps},v).
$$
Since $z|_{\eps=0}=z$ is a minimizer, computing the first variation
$\frac{d}{d\eps}\left(\mathcal F_h(z_{\eps})\right)_{\eps=0}=0$ will classically
give \eqref{eq:discrete_Euler-Lagrange_u}. Similarly considering
$v_{\eps}={(\Phi_{\eps})}_{\#}v$ and $z_{\eps}=(u,v_{\eps})$ will
produce \eqref{eq:discrete_Euler-Lagrange_v}.

More precisely, differentiating the Wasserstein distance squared, the confining potential, and the
diffusive energy are by now classical computations \cite{MR2401600}.
However, differentiating the coupling energy is quite delicate here: because we have
to consider separate horizontal and vertical perturbations the nonscalar nature of the
problem induces a loss of symmetry. Formally the result should follow from
$$
\int_{\R^d}|\nabla\psi_{\eps}|^2\mathrm{d}x=\int_{\R^d}\psi_{\eps}(u_{\eps}-v)\,\mathrm{d}x=\iint_{\R^d\times\R^d}
[u_{\eps}-v](x)G(x-y)[u_{\eps}-v](y)\,\mathrm{d}x\mathrm{d}y
$$
and the classical computations for interaction energies, see
\cite{MR1964483}. But because we consider two components
independently it might happen that $\nabla\psi_{\eps}\notin$
$L^2(\R^d)$ even though $\nabla\psi\in L^2(\R^d)$, and the above
integration by parts might not be legitimate. Moreover since $\nabla
G$ is more singular than $G$ itself, differentiating with respect to
$\eps$ requires some extra regularity. This can actually be made
rigorous using the propagation of the initial regularity as follows.
Since the initial datum $u^0,v^0\in L^{2d/(d+1)}(\R^d)$ and the time
step is small enough, we have by
Proposition~\ref{prop:discrete_Lp_propagation} that $u,v\in
L^1(\R^d)\cap L^{2d/(d+1)}(\R^d)$, and in particular $u,u_\eps,v\in
L^{2d/(d+2)}(\R^d)$. Using Proposition~\ref{prop:Poisson_IBP} we can
therefore integrate by parts and expand with
$u_{\eps}={(\Phi_{\eps})}_{\#}u$
\begin{align*}
&\int_{\R^d}|\nabla\psi_{\eps}|^2\mathrm{d}x
=\iint_{\R^d\times\R^d}u(x)G(\Phi_{\eps}(x)-\Phi_{\eps}(y))u(y)\,\mathrm{d}x\mathrm{d}y\\
&\qquad-2\iint_{\R^d\times\R^d}u(x)G(\Phi_{\eps}(x)-y)v(y)\,\mathrm{d}x\mathrm{d}y+\text{
terms independent of }\eps,
\end{align*}
where the last equality follows by definition of the pushforward
$u_{\eps}={(\Phi_{\eps})}_{\#}u$. In order to differentiate under
the integral sign we only need $L^1(\R^d\times\R^d)$ bounds such
that
\begin{align*}
\iint_{\R^d\times\R^d} u(x)\Big{|}\langle\nabla
G(\Phi_{\eps}(x)-\Phi_{\eps}(y)),\zeta\circ\Phi_{\eps}(x)
-\zeta\circ\Phi_{\eps}(y)\rangle\Big{|}u(y)\,\mathrm{d}x\mathrm{d}y\leq C,\\
\iint_{\R^d\times\R^d} u(x)\Big{|}\langle\nabla
G(\Phi_{\eps}(x)-y),\zeta\circ\Phi_{\eps}(x)\rangle\Big{|}v(y)\,\mathrm{d}x\mathrm{d}y\leq
C,
\end{align*}
uniformly as $\eps\to 0$. Because $\Phi_{\eps}$ is close to
$\mathrm{Id}$ for small $\eps$,
$\zeta\in\mathcal{C}^{\infty}_0(\R^d)$, and $$|\nabla G(x-y)|\leq
\frac{C}{|x-y|^{d-1}},$$ this simply amounts to controlling
\begin{align*}
\iint_{\R^d\times\R^d}u(x)\frac{1}{|x-y|^{d-1}}u(y)\,\mathrm{d}x\mathrm{d}y\leq C,\\
\iint_{\R^d\times\R^d}u(x)\frac{1}{|x-y|^{d-1}}v(y)\,\mathrm{d}x\mathrm{d}y\leq
C,
\end{align*}
which is valid by \eqref{eq:HLS_nablaG_p} with  $p=2d/(d+1)$ and $u,v\in L^{2d/(d+1)}(\R^d)$.
As a consequence we can legitimately compute with
$\Phi_{0}=\operatorname{Id}$
\begin{align*}
\frac{\mathrm{d}}{\mathrm{d}\eps}\Big(\int_{\R^d}|\nabla\psi_{\eps}|^2\mathrm{d}x\Big)_{\eps=0}
&=\iint_{\R^d\times\R^d}u(x)\langle\nabla
G(x-y),\zeta(x)-\zeta(y)\rangle{u}(y)\,\mathrm{d}x\mathrm{d}y\\
&\qquad-2\iint_{\R^d\times\R^d} u(x)\langle\nabla
G(x-y),\zeta(x)\rangle{v}(y)\,\mathrm{d}x\mathrm{d}y.
\end{align*}
Exploiting the symmetry $\nabla G(x-y)=-\nabla G(y-x)$ we finally
get
\begin{align*}
\frac{\mathrm{d}}{\mathrm{d}\eps}\Big(\frac{1}{2}\int_{\R^d}|\nabla\psi_{\eps}|^2\mathrm{d}x\Big)_{\eps=0}
&=\iint_{\R^d\times\R^d} u(x)\langle\nabla G(x-y),\zeta(x)\rangle[u-v](y)\,\mathrm{d}x\mathrm{d}y\\
  &=\int_{\R^d}u\langle\nabla\psi,\zeta\rangle\,\mathrm{d}x
\end{align*}
as in our claim, and the proof is complete.
\end{proof}
The above restriction at initial data $u^0,v^0\in
L^{2d/(d+1)}(\R^d)$, which then is inherited by the solutions to
later times, is technically essential in order to differentiate
under the integral sign with respect to $\varepsilon$-perturbations
and retrieve the discrete Euler-Lagrange equations. Actually this
restriction is not purely technical: in \eqref{eq:PNP} it seems
natural to require the terms $u\nabla\psi,v\nabla\psi$ to be at
least in $L^1(\R^d)$ at time $t=0$. If $u^0,v^0$ are both in
$L^p(\R^d)$ for some $p$ then the integrability for $\nabla\psi$
coming from \eqref{eq:HLS_nablaG_p} is $\nabla\psi\in
L^{dp/(d-p)}(\R^d)$, which is optimal since HLS inequalities are.
Solving for $p'=\frac{dp}{d-p}$ gives exactly the sharp $p=2d/(d+1)$
exponent. Technically speaking we had to assume initial
$L^{r_0}(\R^d)$ regularity with slightly better but arbitrarily
close $r_0>2d/(d+1)$. This is needed for technical compactness
issues, arising later on when we take the limit $h\to 0$ to retrieve
the weak solution $(u,v)=\lim (u_h,v_h)$.
\label{rmk:initial_integrability}

In addition to being an approximate solution in the sense of the
previous Proposition, the interpolation $(u_h,v_h)$ satisfies
\begin{cor}[continuous gradient estimate]
Fix $m\geq 1$. Then for all $0<h<T$,
\begin{equation}
\big\|\nabla(u_h)^{m/2}\big\|_{L^2(h,T;L^{2}(\R^d))}
+\big\|\nabla(v_h)^{m/2}\big\|_{L^2(h,T;L^{2}(\R^d))}\leq
C(T+1)^{\frac{1}{2}},
\label{eq:interpolated_estimate_time-space_grad_um-vm_Lq}
\end{equation}
and
\begin{equation}
\big\|\nabla(u_h)^m\big\|_{L^2(h,T;L^{1}(\R^d))}
+\big\|\nabla(v_h)^m\big\|_{L^2(h,T;L^{1}(\R^d))}\leq
C(T+1)^{\frac{1}{2}},
\label{eq:interpolated_estimate_time-space_grad_u-v_Lq}
\end{equation}
for some constant $C=C(u^0,v^0)>0$ independent of $h$.
\label{cor:interpolated_estimate_time-space_grad_np}
\end{cor}
\begin{proof}
For once the argument requires no distinction between $m>1$ or
$m=1$. We only estimate the $u$ component because the computations
are identical for $v$.

Since $u_h^{m/2},\nabla u_h^{m/2}\in L^2(\R^d)$ we have
$$
\nabla(u_h)^m=2u_h^{\frac{m}{2}}\nabla\big(u_h^{\frac{m}{2}}\big)\in
L^1(\R^d).
$$
Recalling that $u_h^{m/2}(t)$ is actually bounded in $L^2(\R^d)$
uniformly in $t\geq 0$ and $h$, clearly
\eqref{eq:interpolated_estimate_time-space_grad_u-v_Lq} will follow
from \eqref{eq:interpolated_estimate_time-space_grad_um-vm_Lq} and
we only establish the latter.

For fixed $0<h<T$ let $N=\lfloor T/h\rfloor$, and recall that the
interpolation $z_h(t)$ is piecewise constant. Multiplying
\eqref{eq:discrete_gradient_np^m_estimate} by $h>0$ and summing from
$n=0$ to $n=N$ we obtain
\begin{align}
\int_h^{T}&\big\|\nabla(u_h)^{m/2}\big\|^2_{L^2(\R^d)}\mathrm{d}t\nonumber\\
&\leq\int_h^{(N+1)h}\Big\|\nabla\big(u_h(t)^{m/2}\big)\Big\|^2_{L^2(\R^d)}\mathrm{d}t=\sum\limits_{n=0}^{N-1}h\Big\|\nabla\big(u_h^{(n+1)}\big)^{m/2}\Big\|^2_{L^2(\R^d)}\nonumber\\
&\leq C\sum\limits_{n=0}^{N-1}\left(h+\mathcal{H}(u_h^{(n)})-\mathcal{H}(u_h^{(n+1)})+\mathcal{H}(v_h^{(n)})-\mathcal{H}(v_h^{(n+1)})\right)\nonumber\\
&\leq
C\left(T+\mathcal{H}(u^0)+\mathcal{H}(v^0)-\mathcal{H}\big(u_h^{(N)}\big)-\mathcal{H}\big(v_h^{(N)}\big)\right).
\label{eq:discrete_estimate_gradu_Hboltzmann}
\end{align}
By Proposition~\ref{prop:interpolated_Lm_2nd_moment_estimate} the second moments $\mathfrak{m}_2\left(u^{(n)}_h\right),\mathfrak{m}_2\left(v^{(n)}_h\right)$ are bounded uniformly in $t,h,n$,
hence by the Carleman estimate \eqref{eq:Carlemann} we see that $-\mathcal{H}\big(u_h^{(N)}\big)-\mathcal{H}\big(v_h^{(N)}\big)\leq
C$ in \eqref{eq:discrete_estimate_gradu_Hboltzmann} and the proof is
complete.
\end{proof}
%
We observe that another possible way to retrieve better gradient regularity is to
estimate
$$\left|\int_{\R^d}\operatorname{div}(\zeta)u_h^m\,\mathrm{d}x\right|+\left|\int_{\R^d}\operatorname{div}(\zeta)v_h^m\,\mathrm{d}x\right|\leq
C\|\zeta\|_{L^p(\R^d)}$$ for arbitrary vector-fields $\zeta\in
\mathcal{C}^{\infty}_c(\R^d;\R^d)$ in the Euler-Lagrange equations
\eqref{eq:discrete_Euler-Lagrange_u}-\eqref{eq:discrete_Euler-Lagrange_v},
which would estimate by duality $\nabla u_h^m,\nabla v_h^m\in
L^{p'}(\R^d)$, see e.g. \cite{MR1613500}. This approach would only
improve the previous total variation estimate if $p<\infty$, so that
$(L^{p'}(\R^d))'=L^p(\R^d)$ and $\mathcal{C}_c^{\infty}(\R^d)$ is
dense in $L^{p}$. Unfortunately we are here in a limiting situation
where essentially $u_h \nabla \psi_h,v_h \nabla \psi_h \in
L^1(\R^d)$ only, so this is not feasible. More precisely, our
assumption $u^0,v^0\in L^{r_0}(\R^d)$ with $r_0>2d/(d+1)$ in fact
does give slightly better $u_h \nabla \psi_h \in L^{1+\delta}(\R^d)$
integrability through HLS inequalities (for some $\delta>0$
depending on $r_0$). But since $r_0$ could be arbitrarily close to
the critical $2d/(d+1)$ exponent, $\delta>0$ is arbitrarily small
and we shall refrain from taking this technical path.


%
\section{Convergence to a weak solution}
\label{section:CV_to_weak_solution} This section is devoted to the
convergence of the previously approximated interpolating solution
towards the final weak solution, $(u,v)=\lim\limits_{h\to 0} (u_h,v_h)$
in some suitable topology.
 Because of the quadratic interaction term and nonlinear diffusion if $m>1$ we will need the following strong convergence

\begin{theo}[Strong convergence] There is a discrete subsequence, still denoted $h\searrow 0$, and functions $u,v$ such that
\begin{equation}
\label{eq:uh_vh->u_v_a.e.}
u_h(t,x)\to u(t,x) \text{ and } v_h(t,x)\to v(t,x)\qquad\text{a.e. in }(0,\infty)\times\R^d
\end{equation}
and
\begin{equation}
\label{eq:Lp-Lq_CV}
 \forall \,1\leq p<\infty,\,1\leq q<r_0:\qquad
 u_h,v_h\to u,v\quad \mbox{in }L^p_{loc}([0,\infty);L^q(\R^d)),
\end{equation}
where $r_0>\max\{m,2d/(d+1)\}$ is the initial integrability as in
\eqref{eq:initial_data}. \label{theo:uh->u_a.e.}
\end{theo}
Observe in particular that $q=m$ and $q=2d/(d+1)$ are allowed in
\eqref{eq:Lp-Lq_CV}, which will be crucial in order to pass to the
limit in the Euler-Lagrange equations later on. Roughly speaking,
$L^r(\R^d)$ integrability suffices to guarantee $L^q(\R^d)$
convergence for all $q<r$. Unfortunately $q=\max\{m,2d/(d+1)\}$ is a
borderline case that we could not treat, and this is why we needed
to assume initial $L^{r_0}(\R^d)$ integrability for some slightly
better but arbitrarily close $r_0>\max\{m,2d/(d+1)\}$.

\textit{Strategy of proof:} We will use a compactness
criterion in Bochner spaces from \cite{MR916688} that involves

\noindent (i) boundedness in $L^p(0,T;X)$ for some strong $X$
topology,

\noindent (ii) compactness in $L^p(0,T;Y)$ for a weaker $Y$ space,

\noindent (iii) a target intermediate $L^p(0,T;B)$ space with
embeddings $X\subset \subset B\subset Y$.

Note that $X\subset\subset B$ will be achieved by space difference
 quotients and that the key ingredient to obtain boundedness of $\{u_h,v_h\}_h$
in the strong $X$ topology is the gradient estimate from
Corollary~\ref{cor:interpolated_estimate_time-space_grad_np}.
Compactness in $L^p(0,T;Y)$ will be ensured by an approximated time
equi-continuity in some suitable $W^{-s,r'}(\R^d)$ space.\\

 We first collect
some technical results and then establish Theorem
\ref{theo:uh->u_a.e.}. To begin with, for $m>1$, we set
\begin{equation}
\quad q_m:=1+\frac{1}{m'}=1+\frac{m-1}{m}\in (1,m).
 \label{eq:def_q}
\end{equation}
Then there exists $\theta_m\in (0,1)$ satisfying
\begin{equation}
\frac{1}{q_m}=(1-\theta_m)\frac{1}{1}+\theta_m\frac{1}{m},
\label{eq:def_theta*}
\end{equation}
and we let
\begin{equation}
\label{eq:def_p}
p_m:=\frac{2m}{\theta_m}>1.
\end{equation}
If $\tau_e$ denotes the usual shift operator in space
$$
e\in \R^d:\qquad\tau_e w(x):=w(x-e),
$$
we also define the weighted Nikolsk'ii spaces
\begin{equation*}
\label{eq:def_Nikolskii_X}
X_m:=\Big\{w\in L^{q_m}(\R^d):\,\sup\limits_{e\in
\R^d}\|\tau_ew-w\|_{L^{q_m}}|e|^{-\frac{\theta_m}{m}}<\infty,
\;\int_{\R^d}|x|^{\frac{2}{m'}}|w|^{q_m}\mathrm{d}x<\infty
\Big\}
\end{equation*}
endowed with their natural Banach norms with $\theta_m/m<1$. By the
Riesz-Fr\'echet-Kolmogorov Theorem we have
$$
X_m\subset\subset L^{q_m}(\R^d).
$$
We note that the above choice for $p=p_m,q=q_m,\theta=\theta_m$ is purely
technical so we shall go as little as possible into details
regarding their explicit values.

For the case $m=1$ one should similarly use
$$
  X_1:=\Big\{w\in L^1(\R^d):\,\nabla{w}\in L^1(\R^d),\,\int_{\R^d}|x|^2|w|\,\mathrm{d}x<\infty
  \Big\}\subset\subset L^1(\R^d).
$$
Since the related argument is fairly easy compared to the
nonlinear case, in what follows we will omit the related proof and
focus on the nonlinear diffusion case and henceforth we assume
$m>1$.
%
Compactness in space will be ensured by
\begin{prop}[Compactness in space] Let $p,q,\theta,X$ as in \eqref{eq:def_q}-\eqref{eq:def_Nikolskii_X},
and fix $T>0$. Then for $h>0$ small enough we have
$$
\|u_h\|_{L^p(h,T;X_m)}+\|v_h\|_{L^p(h,T;X_m)}\leq C_T
$$
uniformly in $h$. \label{prop:Nikolskii_X_bound}
\end{prop}
%
\begin{proof}
For simplicity we write here $p=p_m,q=q_m,\theta=\theta_m, X=X_m$.
We first claim that
\begin{equation}
 \|u_h(t)\|_{X}\leq C\Big(1+\|\nabla
u_h^m(t)\|^{\theta/m}_{L^1(\R^d)}\Big), \ \text{for} \ t \geq h.
\label{eq:pointwise_X_estimate}
\end{equation}
Indeed since $q\in (1,m)$, it follows immediately from
\eqref{eq:interpolated_Lm_estimate} that
$$
\|u_h\|_{L^{\infty}(0,\infty;L^q(\R^d))}\leq C.
$$
Using \eqref{eq:interpolated_Lm_estimate},
\eqref{eq:interpolated_moment_estimate}, and H\"older inequality we
estimate with $q=q_m=1+1/m'$
\begin{align*}
\int_{\R^d}|x|^{\frac{2}{m'}}u_h^q(t)\,\mathrm{d}x
&=\int_{\R^d}\left(u_h(t)|x|^2\right)^{\frac{1}{m'}}u_h(t)\,\mathrm{d}x\\
& \leq \Big(\int_{\R^d}
|x|^2u_h(t)\,\mathrm{d}x\Big)^{\frac{1}{m'}}\|u_h(t)\|_{L^m(\R^d)}\leq
C.
\end{align*}
Fixing $e\in \R^d$ and using the convexity inequality $|a-b|^m\leq
\big| |a|^m-|b|^m\big|$ for $a,b\geq 0$, we get
\begin{align*}
\big\|\tau_e u_h(t)-u_h(t)\big\|_{L^m(\R^d)}
&=\Big(\int_{\R^d}|\tau_e u_h(t)-u_h(t)|^m\mathrm{d}x\Big)^{\frac{1}{m}}\\
&\leq\Big(\int_{\R^d}|\tau_eu_h^m(t)-u_h^m(t)|\,\mathrm{d}x\Big)^{\frac{1}{m}}\\
&\leq|e|^{1/m}\big\|\nabla u_h^m(t)\big\|_{L^1(\R^d)}^{1/m}.
\end{align*}
By \eqref{eq:def_theta*} and
$\|\tau_eu_h(t)-u_h(t)\|_{L^1(\R^d)}\leq 2$ we get by interpolation
\begin{align*}
\big\|\tau_e u_h(t)-u_h(t)\big\|_{L^q(\R^d)}
&\leq\big\|\tau_e u_h(t)-u_h(t)\big\|_{L^1(\R^d)}^{1-\theta}\cdot\big\|\tau_e u_h(t)-u_h(t)\big\|_{L^m(\R^d)}^{\theta}\\
  & \leq 2|e|^{\frac{\theta}{m}}\big\|\nabla u_h^m(t)\big\|_{L^1(\R^d)}^{\frac{\theta}{m}},
\end{align*}
thus \eqref{eq:pointwise_X_estimate} holds as claimed.

Taking now the $L^p(h,T)$ norm with $p=p_m=2m/\theta$ in
\eqref{eq:pointwise_X_estimate} and using
Corollary~\ref{cor:interpolated_estimate_time-space_grad_np} finally
leads to
$$
\|u_h\|_{L^p(h,T;X)}\leq C\left(T+\big\|\nabla
u_h^m\big\|^{2/p}_{L^2(h,T;L^1(\R^d))}\right)\leq C_T.
$$
The estimate for the $v$ component is again identical.
\end{proof}
Next, we turn to compactness in time in a weaker topology. We first
have
\begin{prop}[Time-equicontinuity in $W^{-s,r'}(\mathbb{R}^d)$]
Let $s,r>0$ be large enough so that
\begin{equation*}
W^{s,r}(\R^d)\subset W^{1,2m'}(\R^d)\cap W^{2,\infty}(\R^d).
\end{equation*}
Then
 \begin{equation}
 \big\|u_h(t_2)-u_h(t_1)\big\|_{W^{-s,r'}}+\big\|v_h(t_2)-v_h(t_1)\big\|_{W^{-s,r'}}\leq C\sqrt{|t_2-t_1|+h}, \ 0\leq t_1\leq t_2,
  \label{eq:1/2_Holder_H-alpha}
 \end{equation}
for some $C=C_{s,r}>0$ independent of $t_1,t_2$ and $h$.
\label{prop:1/2_Holder_H-alpha}
\end{prop}

\begin{proof}
The argument is very similar to \cite[Lemma 13]{MR3040679},
beginning with the calculation of the approximate 1/2-H\"older
continuity of the sequence  $\{u_h(t)\}$. For indices  $0< n < n'$ let
$\nabla q$  denote the optimal map from $u_h^{(n)}$ to $u_h^{(n')}$
so that   $u_h^{(n')} = (\nabla q)_{\#}u_h^{(n)}$. Then
\begin{equation*}
\label{eq:5.2}
\int_{\R^d}\left(u_h^{(n')} - u_h^{(n)}\right)\xi \mathrm{d}x =
\int_{\R^d} \left(\xi(\nabla q(x)) - \xi(x)\right)u_h^{(n)}(x)
\mathrm{d}x, \ \xi\in\mathcal{C}^{\infty}_c(\R^d).
\end{equation*}
Expanding the integrand on the right,
\begin{align*}
\xi(x)-\xi(\nabla q(x))&=\nabla\xi(\nabla q(x))\cdot(x-\nabla q(x))\\
&\qquad+\mathcal{O}\big(|x-\nabla
q(x)|^2\|\nabla^2\xi\|_{L^\infty(\R^d)}\big),
\end{align*}
Hence
\begin{align*}
\int_{\R^d} \left( u^{(n')}_h-u^{(n)}_h \right) \xi\,\mathrm{d}x
  & = \int_{\R^d} \left[\xi\circ\nabla q -\xi\right]u_h^{(n)}\,\mathrm{d}x\nonumber\\
&=\int_{\R^d}\langle\nabla q -\mathrm{Id},\nabla\xi\circ\nabla
q\rangle u^{(n)}_h\,\mathrm{d}x\nonumber\\
&\qquad+\mathcal{O}\big(\|\nabla^2\xi\|_{L^\infty(\R^d)}\big)\int_{\R^d}|\mathrm{Id}-\nabla q|^2 u^{(n)}_h\mathrm{d}x\nonumber\\
& =\int_{\R^d}\langle\nabla q-\mathrm{Id},\nabla\xi\circ\nabla q\rangle{u}^{(n)}_h\,\mathrm{d}x\nonumber\\
&\qquad+\mathcal{O}\left(\| \nabla^2\xi\|_{L^\infty(\R^d)}d_W\big(u^{(n')}_h,u^{(n)}_h\big)^2\right).
\end{align*}
We further compute by Cauchy-Schwarz and H\"older inequalities
\begin{align*}
&\left|\int_{\R^d} \langle\nabla q -\mathrm{Id},\nabla\xi\circ\nabla q\rangle u^{(n)}_h\,\mathrm{d}x \right| \\
&\leq \left(\int_{\R^d} \left|\nabla q-\mathrm{Id}\right|^2
u^{(n)}_h\,\mathrm{d}x\right)^{\frac{1}{2}}\cdot
\left(\int_{\R^d}\left|\nabla\xi\circ\nabla q\right|^2u^{(n)}_h\,\mathrm{d}x\right)^{\frac{1}{2}}\\
&\leq d_W\big(u_h^{(n)},u_h^{(n')}\big)
\Big( \int_{\R^d} | \nabla \xi|^2  u_h^{(n')} \, \mathrm{d}x\Big)^{\frac{1}{2}}\\
&\leq
d_W\big(u^{(n)}_h,u^{(n')}_h\big)\big\|u_h^{(n')}\big\|_{L^m(\R^d)}^{\frac{1}{2}}
\big\| |\nabla\xi|^2\big\|^{\frac{1}{2}}_{L^{m'}(\R^d)}\nonumber\\
& \leq  C \,d_W\big(u^{(n)}_h,u^{(n')}_h\big)||
\nabla\xi||_{L^{2m'}(\R^d)},
\end{align*}
Given $0 < t_1 < t_2$  and $N_1 =\lfloor t_1/h\rfloor , N_2 =\lfloor t_2/h\rfloor$, from \eqref{eq:discrete_square_distance_estimate} and the Cauchy Schwarz inequality we get
\begin{equation*}
 d_W(u^{N_1}_h,u^{N_2}_h) \leq \sum\limits_{n=N_1}^{N_2-1}d_W\big(u^{(n)}_h,u^{(n+1)}_h\big)
 \leq C\sqrt{|t_2-t_1|+h},
 \label{eq:discrete one half Holder}
\end{equation*}
and then
\begin{align*}
& \left|\int_{\R^d}\big(u_h(t_2)-u_h(t_1)\big)\xi\,\mathrm{d}x\right|\\
  &\qquad \leq C\left(\|\nabla\xi\|_{L^{2m'}(\R^d)}\sqrt{|t_2-t_1|+h}+\| \nabla^2\xi\|_{L^\infty(\R^d)}h\right).
\end{align*}
With our choice $W^{s,r}\subset W^{1,2m'}\cap W^{2,\infty}$ and
because $h$ is small we finally obtain
\begin{align*}
\left|\int_{\R^d}\big(u_h(t_2)-u_h(t_1)\big)\xi\,\mathrm{d}x\right|
&\leq C\left(\sqrt{|t_2-t_1|+h}+h\right)\|\xi\|_{W^{s,r}}\\
&\leq C\sqrt{|t_2-t_1|+h}\|\xi\|_{W^{s,r}}.
\end{align*}
Our statement follows by density of $\mathcal{C}^{\infty}_c(\R^d)$ in
$W^{s,r}(\R^d)$ and duality $\left(W^{s,r}(\R^d)\right)'=W^{-s,r'}(\R^d)$.

\end{proof}
We are now in position to prove the desired convergence when $h\to
0$:
\begin{proof}[Proof of Theorem~\ref{theo:uh->u_a.e.}]
Once again we only establish the result for the $u$ component. Fix
any $0<\delta<T$ and let $q=q_m,\theta=\theta_m,p=p_m,X=X_m$ as in
\eqref{eq:def_q}-\eqref{eq:def_Nikolskii_X}. Taking $s,r$ large
enough such that
$$
W^{s,r}(\R^d)\subset\subset L^{m'}_{loc}(\R^d)
$$
is compact. By truncation, a standard duality argument then ensures
that
$$
L^{m}(\R^d)\cap L^1(\R^d)((1+|x|^2)\mathrm{d}x)\subset\subset
W^{-s,r'}(\R^d)
$$
is also compact. By Proposition
\ref{prop:interpolated_Lm_2nd_moment_estimate}, we see that there is
a fixed $W^{-s,r'}(\R^d)$-relatively compact set $K$ such that
$u_h(t)\in K$ for all $t\geq 0$ and small $h>0$. Therefore, we infer
from Proposition~\ref{prop:1/2_Holder_H-alpha} that a refined
version of Arzel\`a-Ascoli Theorem
 \cite[Proposition 3.3.1]{MR2401600} can be applied to conclude that
there exists $u\in\mathcal{C}([0,T];W^{-s,r'}(\R^d))$ such that
$$
\forall \,t\in[0,T],\qquad u_h(t)\to u(t)\qquad\text{in } \;\;
W^{-s,r'}(\R^d)
$$
for some (discrete) subsequence $h\searrow 0$, not relabeled here
for simplicity.
This pointwise convergence together with the uniform
$L^{m}(\R^d)\cap L^1(\R^d,(1+|x|^2)\mathrm{d}x)$ bounds and
Lebesgue's Dominated Convergence Theorem  therefore guarantee strong
convergence
\begin{equation}
u_h\to u\text{ in }L^{p}(0,T;W^{-s,r'}(\R^d)).
\label{eq:CV_stronh_H-alpha}
\end{equation}
By diagonal extraction we can moreover assume that $u\in
\mathcal{C}([0,\infty);W^{-s,r'})$ and that
\eqref{eq:CV_stronh_H-alpha} holds for all $T>0$.

Choosing $s,r$ large enough we can further assume that
$$
  X\subset\subset L^{q}(\R^d)\subset W^{-s,r'}(\R^d).
$$
We recall from Proposition~\ref{prop:Nikolskii_X_bound} that
$\{u_h\}_h$ is bounded in $L^p(\delta,T;X)$, and by
\eqref{eq:CV_stronh_H-alpha} it is also relatively compact in
$L^p(\delta,T;W^{-s,r'}(\R^d))$. By \cite[Lemma 9]{MR916688} we conclude
that $\{u_h\}_h$ is relatively compact in the intermediate
target space, i-e $u_h\to u$ in $L^p(\delta,T;L^q(\R^d))$ for some
subsequence. By a diagonal extraction we may assume that $u$ is
independent of $\delta,T$,
 and $u_h\to u$ in $L^p_{loc}(0,\infty;L^{q}(\R^d))$. Up to extraction of a further
 subsequence this classically implies the desired pointwise convergence a.e.
 $(t,x)\in (0,\infty)\times \R^d$.

 Let us turn now to the $L^p_{loc}([0,\infty);L^q(\R^d))$ convergence, and fix $1\leq p <\infty$ and $1\leq q<r_0$ as in our statement (we recall here that $r_0>\max\{m,2d/(d+1)\}$ is the initial integrability $u^0,v^0\in L^{r_0}(\R^d)$). Once again we only focus on the $u$ component. Thanks to the previous $(t,x)$ a.e. convergence we shall apply Vitali's convergence theorem in fixed bounded intervals $(0,T)$, and we only need to check that the sequence $\{u_h\}$ is tight and uniformly integrable in time and space. Tightness in time is obvious in bounded intervals, and tightness in space is easily obtained by Young's inequality
 $$
 \forall\, t\geq 0:\qquad \int_{\R^d}u_h^q |x|^{2\eps}\mathrm{d}x =\int_{\R^d}\underbrace{u_h^{q-\eps}}_{\in L^{(1/\eps)'}}\underbrace{\left(u_h |x|^{2}\right)^{\eps}}_{\in L^{1/\eps}}\mathrm{d}x\leq C
 $$
 uniformly in $h,t$ for some suitably small $\eps>0$. Here we used the uniform bounds on the second moment $\mathfrak{m}_2(u_h)\leq C$ and $q-\eps \leq q <r_0$ to control $\|u_h(t)^{q-\eps}\|_{L^{(1/\eps)'}(\R^d)}$ by local uniform bounds $\|u_h(t)\|_{L^{r_0}(\R^d)}\leq C_T$ (obtained by propagation of initial integrability, Proposition~\ref{prop:continuous_Lp_propagation}). The same propagation of integrability gives uniform bounds $\|u_h\|_{L^{\infty}(0,T);L^{r_0}(\R^d)}\leq C$, thus by immediate $L^1L^{r_0}$ interpolation we obtain equi-integrability in the form
 $$
 \|u_h\|_{L^{p+\eps}(0,T;L^{q+\eps}(\R^d))}\leq C_T
 $$
 for $\eps>0$ suitably small (essentially such that $1\leq q+\eps<r_0$). Applying Vitali's convergence theorem gives strong convergence
 $u_h\to u$ in $L^p([0,T);L^q(\R^d))$ as desired and the proof is complete.

\end{proof}
We can now prove our main result. The proof of
Theorem~\ref{theo:main_theo_linear} is identical to that of
Theorem~\ref{theo:main_theo} so we only establish the latter.
\begin{proof}[Proof of Theorem~\ref{theo:main_theo}]
\textbf{Step 1: convergence.} Recall that $u_h(t),v_h(t)\in
\mathcal{K}$ for all $t,h$, and that $\mathcal{K}$ is $L^1$-weak
relatively compact. Using the approximate $1/2$ H\"older
equicontinuity \eqref{eq:1/2_Holder_H-alpha} and applying the
previous refined Arzel\`a-Ascoli theorem, we can extract a
subsequence such that
\begin{equation*}
\forall t\geq 0:\qquad u_h(t),v_h(t)\rightharpoonup u(t),v(t)\text{
in }L^1(\R^d)
 \label{eq:CV_weak_uv_L1}
\end{equation*}
for some $u,v\in \mathcal{C}(0,T;\mathcal{P}(\R^d))$ and
$u(t),v(t)\in \mathcal{K}$ for all times. This entails the
$L^{\infty}(0,T;L^m(\R^d)\cap L^1(\R^d)((1+|x|^2)\mathrm{d}x))$
bounds. By standard truncation arguments we also get that
$u(t),v(t)$ are probability measures for all times, and because
$d_W^2$ is l.s.c for the $L^1$- weak convergence we can
moreover take the limit in
\eqref{eq:interpolated_approximate_Holder_estimate} to deduce that
$t\mapsto u(t),v(t)$ are $1/2$-H\"older continuous in
$(\mathcal{P},d_W)$. Since $u_h(0),v_h(0)=u^0,v^0$ we
can take the limit $u(0),v(0)=u^0,v^0$, which together with $u,v\in
\mathcal{C}^{1/2}([0,\infty);\mathcal{P})$ shows that the limit
$u,v$ satisfies the initial condition at least in the sense of
measures as desired.

We claim now that
\begin{equation}
 \nabla u_h^m,\nabla v_h^m\rightharpoonup\nabla u^m,\nabla v^m\quad\text{in }L^{2}(\delta,T;L^1(\R^d))
\label{eq:CV_NL_gradients_weak}
\end{equation}
for all $0<\delta<T$, and also
$$
\nabla u^m,\nabla v^m\in L^2(0,T;L^1(\R^d)).
$$
To see this fix a test function $\varphi\in L^{2}(\delta,T;L^{\infty}(\R^d))$, and write for small $h>0$
$$
\int_{\delta}^T\int_{\R^d}(\nabla
u_h^m)\varphi\,\mathrm{d}x\mathrm{d}t=
\int_{\delta}^T\int_{\R^d}2\nabla
u_h^{m/2}\big(u_h^{m/2}\varphi\big)\,\mathrm{d}x\mathrm{d}t.
$$
By Corollary~\ref{cor:interpolated_estimate_time-space_grad_np} and pointwise a.e convergence $u_h\to u$ we
can assume $\nabla u_h^{m/2}\rightharpoonup\nabla u^{m/2}$ in
$L^2(\delta,T;L^2(\R^d))$ for fixed $0<\delta<T$, and by diagonal extraction we can assume that the limit $\nabla u^{m/2}$ is independent of $\delta,T$. By \eqref{eq:Lp-Lq_CV} with $q=m<r_0$
it is easy to get $u_h^{m/2}\varphi\to u^{m/2}\varphi$ in $L^2(\delta,T;L^2(\R^d))$.
As a consequence we can pass to the limit
\begin{equation*}
\int_{\delta}^T\int_{\R^d}(\nabla
u_h^m)\varphi\,\mathrm{d}x\mathrm{d}t\underset{h\searrow 0}{\to} 2
\int_{\delta}^T\int_{\R^d}\nabla
u^{m/2}\big(u^{m/2}\varphi\big)\,\mathrm{d}x\mathrm{d}t=\int_{\delta}^T\int_{\R^d}\nabla
u^m\varphi\,\mathrm{d}x\mathrm{d}t
\end{equation*}
to obtain \eqref{eq:CV_NL_gradients_weak}. In particular by
Corollary~\ref{cor:interpolated_estimate_time-space_grad_np} we see
that $\forall0<\delta<T$, it holds
\begin{align*}
\big\|\nabla u^{m}\big\|_{L^2(\delta,T;L^1(\R^d))}
&\leq\liminf\limits_{h\searrow 0}\big\|\nabla u_h^{m}\big\|_{L^2(\delta,T;L^1(\R^d))}\\
&\leq 2\liminf\limits_{h\searrow 0}\big\|u_h^{m/2}\big\|_{L^{\infty}(\delta,T;L^2(\R^d))}\big\|\nabla u_h^{m/2}\big\|_{L^2(\delta,T;L^2(\R^d))}\\
&\leq C(1+T)^{1/2}
\end{align*}
uniformly in $\delta>0$, whence $\nabla u^m\in L^2(0,T;L^1(\R^d))$
for all $T>0$.

For the drift terms, recall from \eqref{eq:structural_potentials} that $\nabla U(x),\nabla V(x)$ are at least locally bounded. From Theorem~\ref{theo:uh->u_a.e.} we have $u_h,v_h\to u,v$ at least in $L^1(0,T;L^1(\R^d))$, thus $u_h\nabla U,v_h\nabla V\rightharpoonup u\nabla U,v\nabla V$ in $L^1(0,\infty;L^1_{loc}(\R^d))$ when tested with compactly supported functions $\varphi(x)$. Moreover by uniform bounds on the second moments and linear behavior of $\nabla U,\nabla V$ it is easy to check that the limit $u\nabla U,v\nabla V\in L^\infty(0,\infty;L^1(\R^d))$.

Regarding now the coupling terms $u_h\nabla \psi_h,v_h\nabla\psi_h$, note from
\eqref{eq:Lp-Lq_CV} with $q=2d/(d+1)<r_0$ that we have in particular
$u_h,v_h\to u,v$ in $L^p(\delta,T;L^{2d/d+1}(\R^d))$ for all
$p\in[1,\infty)$.
By strong $L^{2d/(d+1)}(\R^d)\to L^{2d/(d-1)}(\R^d)$ continuity in \eqref{eq:HLS_nablaG_p} we thus obtain $\nabla\psi_h=(\nabla
G)*[u_h-v_h]\to(\nabla G)*[u-v]=\nabla\psi$ in
$L^{p}(\delta,T;L^{2d/(d-1)}(\R^d))$ for all $p\in[1,\infty)$, so by
H\"older inequality
\begin{equation*}
u_h\nabla\psi_h,v_h\nabla\psi_h\to
u\nabla\psi,v\nabla\psi\qquad\text{in }L^{p}(\delta,T;L^1(\R^d))
\end{equation*}
for all $p\in[1,\infty)$ and $0<\delta<T$.
Using the
$L^{\infty}(\delta,T;L^{2d/(d+1)}(\R^d))$ bounds for $u_h,v_h$ this
gives $L^{p}(\delta,T;L^1(\R^d))$ bounds uniformly in $p\geq 1$ and
$\delta>0$, thus $u\nabla\psi,v\nabla\psi\in
L^{\infty}(0,T;L^1(\R^d))\subset L^2(0,T;L^1(\R^d))$ for all $T>0$.

\noindent\textbf{Step 2: the weak solution.} Fix any test-function
$\varphi\in\mathcal{C}^{\infty}_c(\R^d)$ and $0<T_1<T_2$, and let
$N_1=\lfloor T_1/h\rfloor,N_2=\lfloor T_2/h\rfloor$. Let
$\nabla{q}^{(n)}$ be the optimal map in
$u_h^{(n+1)}=\big(\nabla{q}^{(n)}\big)_{\#}u_h^{(n)}$. Expanding
\begin{align}
\varphi(x)-\varphi(\nabla q^{(n)}(x))&= (\nabla\varphi(\nabla
q^{(n)}(x)))\cdot(x-\nabla q^{(n)}(x))\nonumber\\
&\qquad+\mathcal{O}\big(|x-\nabla
q^{(n)}(x)|^2\|D^2\varphi\|_{L^\infty(\R^d)}\big).
 \label{taylor expansion}
\end{align}
Taking $\zeta=\nabla\varphi$ in the Euler-Lagrange equation
\eqref{eq:discrete_Euler-Lagrange_u}, and summing from $n=N_1$ to
$n=N_2-1$ we compute by
\eqref{eq:discrete_square_distance_estimate},
\eqref{eq:discrete_Euler-Lagrange_u} and \eqref{taylor expansion}
that
\begin{align*}
&\int_{\R^d}\big(u_h(T_2)-u_h(T_1)\big)\varphi\,\mathrm{d}x \\
&=\int_{\R^d}\big(u_h^{(N_2)}-u_h^{(N_1)}\big)\varphi\,\mathrm{d}x\\
&=\sum\limits_{n=N_1}^{N_2-1}\int_{\R^d}\big(u_h^{(n+1)}-u_h^{(n)}\big)\varphi\,\mathrm{d}x\\
&=\sum\limits_{n=N_1}^{N_2-1}\int_{\R^d}\big[\varphi\circ\nabla{q}^{(n)}-\varphi\big]u_h^{(n)}\,\mathrm{d}x\\
&=\sum\limits_{n=N_1}^{N_2-1}\int_{\R^d}\langle\nabla
q^{(n)}-\mathrm{Id},\nabla\varphi\circ\nabla q^{(n)}\rangle
u^{(n)}_h\,\mathrm{d}x \\
&\qquad+\sum\limits_{n=N_1}^{N_2-1}\mathcal{O}\Big(\|D^2\varphi\|_{L^\infty(\R^d)}d_W^2\big(u^{(n)}_h,u^{(n+1)}_h\big)\Big) \\
&=\sum\limits_{n=N_1}^{N_2-1}h\int_{\R^d}\left[
\Delta\varphi\big(u^{(n+1)}_h\big)^m-u^{(n+1)}_h\langle\nabla
U,\nabla\varphi\rangle-u^{(n+1)}_h\langle\nabla
\psi^{(n+1)}_h,\nabla\varphi\rangle\right]\,\mathrm{d}x\\
&\qquad\quad+\mathcal{O}\big(h\|D^2\varphi\|_{L^\infty(\R^d)}\big).
\end{align*}
Integrating by parts and exploiting \eqref{eq:discrete_square_distance_estimate}, at the continuous level this becomes
\begin{align*}
&\mathcal{O}\big(h\|D^2\varphi\|_{L^\infty(\R^d)}\big)+\int_{\R^d}\big(u_h(T_2)-u_h(T_1)\big)\varphi\,\mathrm{d}x\nonumber\\
&=-\int_{N_1h}^{N_2h}\int_{\R^d}\Big[\langle\nabla\big(u_h(t)\big)^m,\nabla\varphi\rangle
+\langle\nabla U,\nabla\varphi\rangle u_h(t) + u_h(t)\langle\nabla
\psi_h(t),\nabla\varphi\rangle\Big]\,\mathrm{d}x\mathrm{d}t.
\end{align*}
By step 1 we can take $h\rightarrow 0$ as
\begin{align*}
\int_{\R^d}\big(&u(T_2)-u(T_1)\big)\varphi\,\mathrm{d}x \\
 &=-\int_{T_1}^{T_2}\int_{\R^d}\big(\langle\nabla
 u^m,\nabla\varphi\rangle
 +u\langle\nabla U,\nabla\varphi\rangle  + u\langle\nabla
 \psi,\nabla\varphi\rangle\big)\,\mathrm{d}x\mathrm d t
\end{align*}
to obtain \eqref{eq:def_weak_formulation_u}.
 The equation for $v$ is similarly obtained.

\noindent\textbf{Step 3: Energy bounds and further regularity}. We first establish the energy bound \eqref{eq:energy_monotonicity}. Arguing as in step 1, using \eqref{eq:Lp-Lq_CV} with $q=2d/(d+2)<r_0$, and \eqref{eq:HLS_nablaG_p} it is easy to conclude that
$\nabla\psi_h\to\nabla\psi$ in $L^p_{loc}(0,\infty;L^2(\R^d))$ for all
$p\in[1,\infty)$. In particular $\nabla\Psi_h(t)\to\nabla\Psi(t)$ in $L^2(\R^d)$ and
$$
\mathcal E_{\rm{cpl}}(u_h(t),v_h(t))\to \mathcal E_{\rm{cpl}}(u(t),v(t))\qquad \mbox{a.e. }t\geq 0.
$$
From the energy control $\|\nabla\psi_h\|_{L^{\infty}(0,\infty;L^2(\R^d))}\leq C$ we also
have $\nabla\psi\in L^{\infty}(0,\infty;L^2(\R^d))$.
Similarly, from the $L^p_{loc}(0,\infty;L^q(\R^d))$ with $q=m<r_0$ it is easy to get
$$
\mathcal E_{\rm{diff}}(u_h(t),v_h(t))\to \mathcal E_{\rm{diff}}(u(t),v(t))\qquad \mbox{a.e. }t\geq 0.
$$
For the potential energy, we have $0\leq u_h(t,x)U(x)\to u(t,x)U(x)$ a.e. $x\in \R^d$ for a.e. fixed $t\geq 0$, and similarly for $v_hV$. By Fatou's lemma we conclude that
$$
\mathcal E_{\rm ext}(u(t),v(t))\leq \liminf\limits_{h\to 0}\mathcal E_{\rm ext}(u_h(t),v_h(t))\qquad \mbox{a.e. }t\geq 0.
$$
Summing $\mathcal E=\mathcal E_{\rm diff}+\mathcal E_{\rm ext}+\mathcal E_{\rm cpl}$ with $\mathcal E(u_h(t),v_h(t))\leq \mathcal E(u^0,v^0)$ at the discrete level finally entails the desired energy bound \eqref{eq:energy_monotonicity}. Note that the potential energy $\mathcal E_{\rm ext}$ is the most problematic term, because we cannot a priori conclude equality $\mathcal E_{\rm ext}(u(t),v(t))=\lim \mathcal  E_{\rm ext}(u_h(t),v_h(t))$ in the last display (in fact strong convergence would imply convergence of the second moments, thus convergence in $d_{W}$). As a consequence we only retrieve the one-sided inequality, and we are unable to conclude that the total energy is monotone nonincreasing in the limit.

Turning now to the
propagation of initial regularity \eqref{eq:Lp_Linfty_estimate},
assume that the initial datum $u^0,v^0\in L^{p}(\R^d)$ for some
$p\in[1,\infty]$. If $p<\infty$ then
Proposition~\ref{prop:continuous_Lp_propagation} bounds
$u_h(t),v_h(t)$ in $L^p(\R^d)$ uniformly in $h$ with exponential control
$$
\sup\limits_{t\in[0,\tau]}\left(\|u_h(t)\|_{L^p(\R^d)}+\|v_h(t)\|_{L^p(\R^d)}\right) \leq
Ce^{\lambda\tau}\left(\|u^0\|_{L^{p}(\R^d)}+\|v^0\|_{L^{p}(\R^d)}\right).
$$
Up to extraction of a further subsequence we can assume that $u_h,v_h\overset{*}{\rightharpoonup} u,v$ in $L^{\infty}_{loc}([0,\infty);L^p(\R^d))$, which immediately gives \eqref{eq:Lp_Linfty_estimate}. If now $u^0,v^0\in L^{\infty}(\R^d)$ clearly \eqref{eq:Lp_Linfty_estimate} holds
for arbitrarily large $p$. Our claim then easily follows by letting
$p\to \infty$ and the proof is achieved.

\end{proof}

%

\section{Appendix} \label{annex}
For $p>1$ let $L^p_w(\R^d)$ be the weak-$L^p$ spaces, which coincide with the usual Lorentz space $L^{p,\infty}(\R^d)$. The natural Banach norm is
\begin{equation}
\|w\|_{L^p_w(\R^d)}=\|w\|_{L^{p,\infty}(\R^d)}=\sup\limits_{t>0}\{t^{1/p}w^*(t)\},
\label{eq:def_Lpweak}
\end{equation}
where $w^*(t)$ is the symmetric-decreasing rearrangement of $w(x)$.
\begin{prop}
Denoting $\Phi=(-\Delta)^{-1}w=G*w$, the Dirichlet energy
$$
w\in L^1(\R^d)\mapsto \mathcal E_D(w)=\int\limits_{\R^d}|\nabla\Phi|^2\,\mathrm{d}x\in[0,+\infty]
$$
is lower semi-continuous for weak $L^1$ convergence.
\label{prop:Dirichlet_lsc}
\end{prop}
\begin{proof}
Let $w_n \rightharpoonup w$ in $L^1(\R^d)$. If $\liminf \mathcal E_D(w_n)=+\infty$ our statement is trivial, so up to extraction of a subsequence we may assume that $\liminf \mathcal E_D(w_n) = \lim \mathcal  E_D(w_n) = C<+\infty$, in particular we have that
\begin{equation}
\lim \|\nabla\Phi_n\|_{L^2(\R^d)}^2=\liminf \mathcal E_D(w_n)<+\infty.
\label{eq:annex_nabla_Phi_leq_C}
\end{equation}
Now since $w_n\rightharpoonup w$ in $L^1(\R^d)$ we see that $w_n$ is bounded in $L^1(\R^d)$, hence by\eqref{eq:HLS_G_1},
$$
\|\Phi_n\|_{L^{d/(d-2)}_w(\R^d)}\leq C\|w_n\|_{L^1(\R^d)}\leq C.
$$
Since
$L^{d/(d-2)}_w(\R^d)=L^{d/(d-2),\infty}(\R^d)=\left(L^{d/2,1}(\R^d)\right)'$
is a topological dual we can also assume, by the Banach-Alaoglu
theorem, and up to a further subsequence, that
$$
\Phi_n\overset{*}{\rightharpoonup}\Phi \quad\text{in }L^{d/(d-2)}_w(\R^d).
$$
By \eqref{eq:annex_nabla_Phi_leq_C} and up to a subsequence we see that
$$
\|\nabla \Phi\|_{L^2(\R^d)}^2\leq \liminf \|\nabla \Phi_n\|_{L^2(\R^d)}^2=\liminf \mathcal E_D(w_n).
$$
As a consequence it suffices to prove that $\Phi=G*w$, since then $\mathcal E_D(w)=\|\nabla \Phi\|_{L^2(\R^d)}^2\leq \liminf \mathcal E_D(w_n)$.

Set $\tilde{\Phi}=G*w\in L^{d/(d-2)}_w(\R^d)$ and let us prove that $\Phi-\tilde{\Phi}=0$. Since $-\Delta \Phi_n =w_n\rightharpoonup w=-\Delta \tilde{\Phi}$ in $L^1(\R^d)$ we have in particular that $\Phi-\tilde{\Phi}$ is harmonic. Because harmonic tempered distributions are polynomials and $L^{d/(d-2)}_w(\R^d)\subset \mathcal{S}'(\R^d)$ we get that $\Phi-\tilde{\Phi} \in L^{d/(d-2)}_w(\R^d)$ is polynomial. By \eqref{eq:def_Lpweak} we see that the polynomial $\Phi-\tilde{\Phi}$ decays at infinity, hence $\Phi-\tilde{\Phi}=0$ as claimed and the proof is complete.
\end{proof}
\subsection*{Acknowledgments}
We would like to thank Adrien Blanchet for fruitful discussions and
also his hospitality to LM. We are very grateful to the careful
reading and important suggestions from the anonymous referee. We
thank Marcus Wunsch for introducing us to this problem. This work is
partially supported by DMS 0806703, DMS 0635983, DMS 091501,
DMS-1217066, DMS 1419053, OISE 0967140, Portugal/CMU program, and
FCT SFRH/BPD/88207/2012.

%
%
\bibliographystyle{plain}
\bibliography{./PNP_new_4.0_bib}
\end{document}